\documentclass[letterpaper,oneside, reqno]{amsart}
\usepackage{mathtools, graphicx, amsthm, amsfonts, amssymb}
\usepackage{breqn}
\usepackage[pdftex]{pict2e}
\usepackage[dvipsnames]{xcolor}
\usepackage[utf8]{inputenc}
\usepackage{times}
\usepackage[T1]{fontenc}
\usepackage{amsrefs,amsmath}
\usepackage{microtype}
\usepackage[letterpaper]{geometry}
\usepackage{hyperref}
\usepackage{cleveref}
\usepackage{tikz}
\usepackage{tikz-cd}
\usepackage{fancyhdr}
\usepackage{tikz-3dplot}
  \usetikzlibrary{positioning,calc}
 \usetikzlibrary{intersections}
 
\usepackage[final]{fixme}
\fxsetup{theme=color, mode=multiuser, noinline}
\FXRegisterAuthor{sa}{SA}{\color{red}Sameer}
\FXRegisterAuthor{rt}{RT}{\color{blue}Rekha}
\FXRegisterAuthor{ae}{AE}{\color{orange}Alperen}
\FXRegisterAuthor{ec}{EC}{\color{purple}Erin}
\FXRegisterAuthor{cv}{CV}{\color{teal}Cynthia}

\theoremstyle{plain}
\newtheorem{thm}{Theorem}[section]
\newtheorem{cor}[thm]{Corollary}
\newtheorem{lem}[thm]{Lemma}

\newtheorem{defn}[thm]{Definition}
\newtheorem{prob}[thm]{Problem}

\theoremstyle{definition}

\theoremstyle{remark}
\newtheorem{rem}[thm]{Remark}
\newtheorem{ex}[thm]{Example}

\newcommand{\RR}{\mathbb R}
\newcommand{\PP}{\mathbb P}
\newcommand{\CC}{\mathbb C}

\newcommand{\rank}{\operatorname{rank}}

\title{The Geometry of Rank Drop in a Class of \\Face-Splitting Matrix Products: Part II}

\author{Erin Connelly}
\address{Department of Mathematics, University of Washington, Seattle}
\email{erin96@uw.edu}
\author{Rekha R. Thomas}
\address{Department of Mathematics, University of Washington, Seattle}
\email{rrthomas@uw.edu}
\author{Cynthia Vinzant}
\address{Department of Mathematics, University of Washington, Seattle}
\email{vinzant@uw.edu}

\begin{document}
\maketitle
\begin{abstract}
Given $k$ points 
$(x_i,y_i) \in \PP^2 \times \PP^2$, we characterize rank deficiency of the $k \times 9$ matrix $Z_k$ with rows $x_i^\top \otimes y_i^\top$ in terms of the geometry of the point configurations 
$\{x_i\}$ and $\{y_i\}$. In~\cite{rankdrop1} we presented results for the cases $k\le 6$. In this paper we deal with the remaining cases $k=7,8$ and $9$. The results involve the interplay of quadric surfaces, cubic curves and Cremona transformations \cite{rankdrop1}. 
\end{abstract}

\section{Introduction}
We are interested in solving the following problem:
\begin{prob} \label{prob:main} 
Given $k \leq 9$ points $(x_i,y_i)\in\PP^2\times\PP^2$, 
consider the $k \times 9$ 
matrix $Z_k$ whose rows are $x^\top_i \otimes y^\top_i$ for $i=1, \ldots, k$, i.e., 
\[
Z_k= \begin{bmatrix} x^\top_1 \otimes y^\top_1 \\ \vdots \\ x^\top_k \otimes y^\top_k \end{bmatrix}.
\]
Delineate the geometry of point configurations  $\{x_i\}$ and $\{y_i\}$ for which $\rank(Z_k) < k$.
\end{prob}
Note that Problem~\ref{prob:main} can be rephrased geometrically and generalized to any algebraic variety.
\begin{prob}\label{prob:geom}
    Given $k \leq 9$ points $(x_i,y_i)\in\PP^2\times\PP^2$, 
    delineate the  geometry of the point configurations  $\{x_i\}$ and $\{y_i\}$ for which the subspace spanned by the images of these points under the Segre embedding of  $\PP^2\times\PP^2$ in $\PP^8$
    has dimension less than $k-1$. 
\end{prob}

Problem~\ref{prob:main} arises in the study of reconstruction problems in 3D computer vision. For background on the problem and related work we direct the reader to part I of this work~\cite{rankdrop1} where Problem~\ref{prob:main} was solved for  $k \leq 6$. The results relied on the classical invariant theory of points in $\PP^2$ and the theory of cubic surfaces. In this paper we complete the characterization for the remaining cases of $k=7,8,9$. Once again, the results can be phrased in terms of classical algebraic geometry and invariants. 

\subsection*{Semi-Genericity}
Throughout this paper, we will concern ourselves with point configurations that are \textit{semi-generic}; a configuration of $k$ point pairs $(x_i,y_i)$ is \textit{semi-generic} if every subset of $k-1$ point pairs is fully generic. That is, we say that a property holds for a semi-generic choice of $(x_i,y_i)\in (\PP^2\times \PP^2)^k$ if there is a nonempty Zariski open set $\mathcal{U}\subseteq (\PP^2\times \PP^2)^{k-1}$ so that the property holds whenever $\{(x_i,y_i):i\neq j\}$ lies in $\mathcal{U}$ for all $j=1,\hdots, k$. 
Despite the name, semi-genericity is actually a stronger notion than usual genericity. 
We use this name because often the property of interest for points in $(\PP^2\times \PP^2)^k$ is that two algebraic conditions coincide, whereas generic points satisfy neither algebraic condition. 
As a small example of this usage, let us instead consider a semi-generic pair of points $x_1, x_2$ in the line $\RR$. Consider $f(x_1,x_2) = x_1(x_2-1)(x_1-x_2)$. Then $f(x_1,x_2)=0$ if and only if $x_1=0$, $x_2=1$, or $x_1=x_2$.  For generic $(x_1,x_2)$, $f(x_1, x_2)\neq 0$. Semi-genericity only allows us to exclude algebraic conditions on $x_1$ and $x_2$ individually. 
In this example, a semi-generic pair of points $(x_1,x_2)$ satisfies $f(x_1,x_2)=0$ if and only if $x_1=x_2$. 
This holds whenever $x_1, x_2\in\mathcal{U} = \RR\backslash\{0,1\}$.

\subsection*{Summary of results and organization of the paper.} 
In \cite{rankdrop1} we studied Problem~\ref{prob:main} algebraically by decomposing the ideal generated by the maximal minors of $Z_k$ into its prime components and examining only those components that did not correspond to rank drop conditions for a submatrix of $Z_k$ with at most $k-1$ rows, called  
{\em inherited conditions}, for the rank deficiency of $Z_k$.  
Through this we obtained both algebraic conditions that completely characterized rank drop, and geometric conditions that characterized rank drop under mild genericity assumptions.
This method cannot be applied to the cases of $k=7,8,9$ due to computational limitations. Additionally, in these cases, the novel component of rank drop has a greater dimension than all the components of inherited conditions. 
Previously, for $k\leq 5$ the novel component had a strictly lower dimension than the variety of inherited conditions, and for $k=6$ the novel component had equal dimension to that of the inherited conditions variety. For this reason, we largely concern ourselves only with the geometric characterization of rank drop for semi-generic configurations with $k=7,8,9$, rather than an algebraic characterization beyond the vanishing of the maximal minors of $Z_k$.

In Section~\ref{sec:background} we establish a number of facts about Cremona transformations, cubic curves, and projective reconstructions that we will use throughout the paper. In Section~\ref{sec:k=8} we study the problem for $k=8$ and prove that $Z_k$ is rank-deficient exactly when there is a quadratic Cremona transformation $f:\PP^2\dashrightarrow\PP^2$ such that $f(x_i)= y_i$ for all $i$ (Theorem~\ref{thm:k=8 main result}). To do so, we establish a correspondence between three sets: lines in the nullspace of $Z_k$, quadrics passing through a projective reconstruction of the input point pairs, and Cremona transformations sending $x_i\mapsto y_i$ (Theorem~\ref{thm:trinity correspondence (vision)} which depends on Theorem~\ref{thm:trinity correspondence}). We refer to this as the {\em trinity correspondence} and it is the foundation for all of our results in this paper. In Section~\ref{sec:k=7} we study the problem for $k=7$ and prove that $Z_k$ is rank-deficient exactly when there are cubic curves in each copy of $\PP^2$, passing through all seven points, and an isomorphism between these curves that sends $x_i\mapsto y_i$ (Theorem~\ref{thm:k=7 simple}). We further prove that this occurs exactly when seven particular cubic curves in each copy of $\PP^2$ are coincident and we provide an algebraic characterization for when this occurs (Theorem~\ref{thm: k=7 central theorem}). In Section~\ref{sec:k=9} we answer Problem~\ref{prob:main} for $k=9$, which is largely straight-forward (Theorem~\ref{thm:k=9}). 
We summarize our results in Section~\ref{sec:conclusion} and state a geometric consequence about reconstructions of semi-generic point pairs of size six, seven and eight.

\subsection*{Acknowledgments} This paper is a sequel to \cite{rankdrop1} written with Sameer Agarwal and Alperen Ergur. 
The third author was partially supported by NSF grant No.~DMS-2153746.
\section{Background and Tools}\label{sec:background}
\subsection{Quadratic Cremona Transformations and Cubic Curves}

\begin{defn}
A {\bf quadratic Cremona transformation} of $\PP^2$ is a birational automorphism $f \,:\, \PP^2 \dashrightarrow \PP^2$ defined as $f(x) = (f_1(x):f_2(x):f_3(x))$ where $f_1, f_2,f_3$ are homogeneous quadratic polynomials in $x = (x_1, x_2,x_3)$.
\end{defn}
We will drop the word ``quadratic'' from now on as all the Cremona transformations we consider will be quadratic. 
Each Cremona transformation can be obtained by blowing up three points $a_1,a_2,a_3$ in the domain (called {\em base points}) at which the transformation is not defined and collapsing three lines $\gamma_1, \gamma_2, \gamma_3$ (called {\em exceptional lines}) which contain pairs of base points: for distinct $i,j,k$, the line $\gamma_i$ contains $a_j,a_k$.  Generically, the base points and exceptional lines of a Cremona transformation will all be distinct; when they are not all distinct, the transformation is said to be degenerate. In this paper we will only consider non-degenerate Cremona transformations.  

The inverse of a Cremona transformation $f$ is also a Cremona transformation with base points $b_1,b_2,b_3$ and exceptional lines $\tau_1, \tau_2, \tau_3$ in the codomain of $f$. The map $f$ sends $\gamma_i \mapsto b_i$ while 
$f^{-1}$ sends $\tau_i \mapsto a_i$. For simplicity we will often refer to both the base points in the domain and the base points in the codomain (i.e. the base points of $f^{-1}$) as the base points of $f$. The standard Cremona transformation is 
\begin{align} \label{eq:std cremona}
    f(x_1,x_2,x_3)=(x_2x_3:x_1x_3:x_1x_2)
\end{align}
which has base points $(1:0:0), (0:1:0),(0:0:1)$ and exceptional lines 
$x_i=0$ for $i=1,2,3$.  This transformation is an involution since it is its own inverse, and the base points and exceptional lines of $f^{-1}$ are again 
$(1:0:0), (0:1:0),(0:0:1)$ and 
$x_i=0$ for $i=1,2,3$. All Cremona transformations differ from the standard one only by projective transformations as stated below.
\begin{lem}
Let $g$ be a Cremona transformation and $f$ be the standard Cremona involution. Then there are projective transformations $H_1,H_2$ such that $g=H_1\circ f\circ H_2$.
\end{lem}

\begin{proof}
 Let $a_1, a_2, a_3 \in \PP^2$ denote the base points of $g$.
The coordinates $(g_1,g_2,g_3)$ of $g$ form a basis for the three-dimensional vector space of quadratics vanishing on the points $a_1,a_2,a_3$. Another basis is $h = (\ell_2 \ell_3, \ell_1\ell_3, \ell_1\ell_2)$ where $\ell_i\in \CC[x,y,z]_1$ defines the line joining $a_j$ and $a_k$ for every labeling $\{i,j,k\} = \{1,2,3\}$.  Therefore there is some invertible linear transformation $H_1$ for which $g = H_1 h$. Similarly, $(\ell_1, \ell_2, \ell_3)$ is a basis for $\CC[x,y,z]_1$ and so there is a linear transformation $H_2$ for which $H_2 (x,y,z) = (\ell_1, \ell_2, \ell_3)$. 
The map $h$ is given by $f\circ H_2$ and so $g =H_1\circ f\circ H_2$.
\end{proof}

Throughout this paper we will be interested in $\PP^2 \times \PP^2$ and we typically denote points in the first $\PP^2$ by $x$ and those in the second $\PP^2$ by $y$. The notation $\PP^2_x$ and $\PP^2_y$ will help keep this correspondence clear.

\begin{lem}\label{lem:basepoints fix monomials to scale}
Let $f \,:\, \PP^2_x \dashrightarrow \PP^2_y$ be a Cremona transformation such that $f$ and $f^{-1}$ have base points $e_1^x=e_1^y=(1:0:0)$, $e_2^x=e_2^y=(0:1:0)$, $e_3^x=e_3^y=(0:0:1)$ in the domain and codomain. Then $f$ has the form
\begin{equation}
    f(x_1,x_2,x_3)=(ax_2x_3: bx_1x_3: cx_1x_2)
\end{equation}
where $a,b,c\in\CC\backslash\{0\}$.
\end{lem} 
\begin{proof}
Suppose $f=(f_1,f_2,f_3)$ where $f_1,f_2,f_3$ are quadratic polynomials. Since $f$ is undefined at the three base points in the domain, it follows that $f_1,f_2,f_3$ contain only the monomials $x_1x_2, x_1x_3, x_2x_3$. Moreover, we know that $f(x_1,x_2,0)=(0:0:1)$. It follows that $f_1,f_2$ do not contain the monomials $x_1x_2$. In examining the other two exceptional lines, we find that $f_1,f_2,f_3$ contain only one monomial each and that $f$ has the desired form.
\end{proof}
We note that the choice of $(a,b,c)$ is equivalent to specifying a single point correspondence $p\mapsto q$, where neither $p$ nor $q$ lie on an exceptional line. It follows that a Cremona transformation has $14$ degrees of freedom: six from the base points in the domain, six from the base points in the codomain, and two from the choice of a single point correspondence.

Next we prove some facts about Cremona transformations and isomorphisms of cubic curves. 

\begin{defn}
Let $f$ be a Cremona transformation with base points $B(f)$. For a 
curve $C\subset\PP^2$, define $f(C) :=\overline{f(C\backslash B(f))}$, and for a given point $p$, let $\nu_p(C)$ be the multiplicity of the curve $C$ at the point $p$. 
\end{defn}

\begin{lem}\label{lem:base points lie on cubic} \cite{diller}
Let $C\subset\PP^2$ be a plane curve of degree $n$ and let $f$ be a Cremona transformation. Then
\begin{equation}
\text{deg}(f(C))=2n-\sum_{p\in B(f)}\nu_p(C).
\end{equation}
In particular, if $C$ is a smooth cubic curve then $f(C)$ is also a cubic curve if and only if the 
base points of $f$ lie on $C$. In this case, $f^{-1}(f(C))=C$ implies that 
the base points of $f^{-1}$ lie on $f(C)$.
\end{lem} 
Using this, we can prove the following result.
\begin{lem}\label{lem:Cremona transformations induce isomorphisms on cubics}
Let $C$ be a smooth cubic curve and let $f$ be a Cremona transformation with base points $a_1,a_2,a_3\in C$ in the domain and $b_1,b_2,b_3$ in the co-domain. Then $f(C)$ is a smooth cubic curve and $\bar{f}:C\to f(C)$, defined by taking the closure of $f|_{C\backslash B(f)}$, is an isomorphism.
\end{lem}
\begin{proof}
By Lemma~\ref{lem:base points lie on cubic}, $f(C)$ is a cubic curve. Moreover, since $f^{-1}(f(C))=C$ is a cubic curve, it also follows that $b_1,b_2,b_3\in f(C)$. The fact that $\bar{f}$ is an isomorphism follows from the corollary after 
\cite[\S 1.6, Theorem 2]{shafarevich} which says that a birational map between nonsingular projective plane curves is regular at every point, and is a one-to-one correspondence.
\end{proof}

Given a smooth cubic curve $C$, any automorphism $g:C\to C$ is of the form $u\mapsto au+b$, where $a=\pm 1$, $b\in C$, where addition is defined via the group law on $C$. Theorem 1.3 in \cite{diller} states that given a smooth cubic curve $C$ and an automorphism $g:C\to C$ defined by some multiplier $a=\pm 1$ and translation $b\in C$, then $g$ is induced by a Cremona transformation with base points $a_1,a_2,a_3$ if and only if $a(a_1+a_2+a_3)=3b$, where again, addition is with respect to the group law on $C$. In particular, every automorphism of $C$ is induced by a two-parameter family of Cremona transformations, which we obtain by picking the first two base points arbitrarily and then letting the third base point be determined by the equation $a_3=a(3b-a_1-a_2)$.

We can use this to prove a converse to Lemma~\ref{lem:Cremona transformations induce isomorphisms on cubics}.
\begin{lem}\label{lem:isomorphisms are cremona}
Let $f:C\to C'$ be an isomorphism of smooth cubic plane curves. Then there is a two-parameter family of Cremona transformations $f'_\sigma:\PP^2\dashrightarrow\PP^2$ such that $f'_\sigma|_C=f$. The base points of these Cremona transformations will lie on the cubic curves.
\end{lem} 
\begin{proof}
Since $C$ and $C'$ are isomorphic, they have the same 
Weierstra{\ss}  form $C_0$. There are therefore homographies $H_1,H_2 \in \textup{PGL}(3)$ such that $H_1(C)=C_0=H_2(C')$ and therefore $H_1^{-1}H_2(C')=C$. Then $H_1^{-1}H_2\circ f:C\to C$ is an automorphism of $C$ and it follows by \cite[Theorem 1.3]{diller} that this is induced by some two-parameter family of Cremona transformations $g_{\sigma}$; 
the members of this family are obtained by picking the first two base points arbitrarily on $C$ and then letting the third base point be determined by the equation $a_3=a(3b-a_1-a_2)$. Then $f'_{\sigma}:=H_2^{-1}H_1\circ g_{\sigma}$ is the desired family of Cremona transformations. By Lemma~\ref{lem:base points lie on cubic} the base points of each of these Cremona transformations lie on the cubic curves.
\end{proof}

\subsection{Fundamental Matrices and Projective Reconstruction}\label{sec:vision stuff}
In this paper we will be concerned with pairs of linear projections $\pi_1,\pi_2\,:\, \PP^3 \dashrightarrow \PP^2$ with non-coincident 
centers $c_1, c_2$. In the context of computer vision, these arise as {\em projective cameras} which are linear projections from $\PP^3(\RR) \dashrightarrow \PP^2(\RR)$, represented by  (unique) matrices $A_1, A_2 \in \PP(\RR^{3 \times 4})$ of 
rank three, such that $\pi_i(p) \sim A_i p$ for all {\em world points} $p \in \PP^3(\RR)$. The notation $\sim$ indicates equality in projective space. The centers $c_i$ are the unique points in $\PP^3(\RR)$ such that 
$A_ic_i = 0$ for $i=1,2$. The projections we consider in this paper are slightly more general in that they work over $\CC$; they 
are represented by rank three matrices $A_i \in \PP(\CC^{3 \times 4})$ and send $p \in \PP^3$ to 
$A_ip \in \PP^2$.

In the vision setting, the image formation equations $A_ip = \lambda_i \pi_i(p)$ for $i=1,2$ and some $\lambda_i \in \RR$, imply that for all  $p \in \PP^3(\RR)$, 
\begin{align}\label{eq:fundamental matrix}
    0 = \det \begin{bmatrix} A_1 & \pi_1(p) & 0 \\ A_2 & 0 & \pi_2(p) \end{bmatrix} = \pi_2(p)^\top F \pi_1(p)
\end{align}
for a unique matrix $F \in \PP(\RR^{3 \times 3})$ of rank two, determined by $(A_1, A_2)$ \cite[Chapter 9.2]{hartley-zisserman-2004}. This matrix $F$ is called the {\em fundamental matrix} of the cameras/projections $(A_1,A_2)$ / $(\pi_1,\pi_2)$. It defines the bilinear form $B_F(x,y) = y^\top F x$
such that 
$B_F(\pi_1(p), \pi_2(p))= \pi_2(p)^\top F \pi_1(p) = 0$ for all $p \in \PP^3(\RR)$
. The entries of $F$ are certain $4 \times 4$ minors of the $6 \times 4$ matrix obtained by stacking $A_1$ on top of $A_2$.
The points $e^x :=\pi_1(c_2)$ and $e^y:=\pi_2(c_1)$ are called the {\em epipoles} of $F$. 
It is well-known \cite[Chapter 9.2]{hartley-zisserman-2004} that $e^x$ and $e^y$ are  the unique points in $\PP^2$ such that $Fe^x=0=(e^y)^\top F$. Conversely, for every rank-two matrix $F\in\PP(\RR^{3\times 3})$ there exists, up to projective transformation, a unique pair of cameras $(A_1,A_2)$ / linear projections $\pi_1,\pi_2  \,:\, \PP^3(\RR) \dashrightarrow \PP^2(\RR)$ with fundamental matrix $F$, see \cite[Theorem 9.10]{hartley-zisserman-2004}. All of these facts extend verbatim over $\CC$ and we call a rank two matrix $F \in \PP(\CC^{3 \times 3})$ a {\em fundamental matrix} of $(\pi_1,\pi_2)$ if it satisfies \eqref{eq:fundamental matrix}.

Equation~\eqref{eq:fundamental matrix} is a constraint on the images of a world point in two cameras. Going the other way, given $k$ point pairs $(x_i,y_i) \in \PP^2(\RR) \times \PP^2(\RR)$, one can ask if they admit a {\em projective reconstruction}, namely a pair of real cameras $A_1, A_2$ and real world points $p_1, \ldots, p_k$ such that $A_1 p_i \sim x_i$ and $A_2 p_i \sim y_i$ for $i=1,\ldots, k$. A necessary condition for a reconstruction is the 
existence of a rank-two matrix $F \in \PP(\RR^{3 \times 3})$ such that $y_i^\top F x_i =0$ for $i=1, \ldots, k$, called a {\em fundamental matrix} of the point pairs 
$(x_i,y_i)_{i=1}^k$. Note that $\textup{vec}(F)$ lies in the nullspace of 
$Z_k = (x_i^\top \otimes y_i^\top)_{i=1}^k$. 
The necessary and sufficient conditions for the existence of a projective reconstruction 
of $(x_i,y_i)_{i=1}^k$ are 1) the existence of a fundamental matrix $F$ and 2) for each $i$, either $Fx_i = 0$ and $y_i^\top F = 0$, or neither $x_i$ nor $y_i$ lie in the right and left nullspaces of $F$ \cite{hon paper}. 
In this paper, we extend the above definition to $\CC$ and call any rank-two matrix $F \in \PP(\CC^{3 \times 3})$ that lies in the 
nullspace of $Z_k$, a {\em fundamental matrix} of the point pairs 
$(x_i,y_i)_{i=1}^k$.


\section{$k=8$}\label{sec:k=8}

In this section we characterize the rank deficiency of $Z = Z_8 = (x_i^\top \otimes y_i^\top)_{i=1}^8$ when the point pairs $(x_i,y_i)$ 
are semi-generic. When $k$ is fixed 
we often write $Z$ instead of $Z_k$.
\begin{thm}\label{thm:k=8 main result}
For eight semi-generic point pairs $(x_i,y_i)_{i=1}^8$, the matrix $Z$ drops rank if and only if there exists a Cremona transformation $f:\PP^2_x\dashrightarrow\PP^2_y$ such that $f(x_i)=y_i$ for all $i$.
\end{thm}
\begin{proof}($\Leftarrow$)
Suppose we have  a Cremona transformation $f:\PP^2_x\dashrightarrow\PP^2_y$ such that $f(x_i)=y_i$ for $i=1,\ldots,8$. After homographies we can assume that $f$ is the basic quadratic involution $(x_1,x_2,x_3)\mapstochar\dashrightarrow(x_2x_3,x_1x_3,x_1x_2)$. Then 
\begin{equation}
Z=\begin{bmatrix}x_{11}x_{12}x_{13} & x_{11}^2x_{13} & x_{11}^2x_{12} & x_{12}^2x_{13} & x_{11}x_{12}x_{13} & x_{11}x_{12}^2 & x_{12}x_{13}^2 & x_{11}x_{13}^2 & x_{11}x_{12}x_{13}\\
x_{21}x_{22}x_{23} & x_{21}^2x_{23} & x_{21}^2x_{22} & x_{22}^2x_{23} & x_{21}x_{22}x_{23} & x_{21}x_{22}^2 & x_{22}x_{23}^2 & x_{21}x_{23}^2 & x_{21}x_{22}x_{23}\\
x_{31}x_{32}x_{33} & x_{31}^2x_{33} & x_{31}^2x_{32} & x_{32}^2x_{33} & x_{31}x_{32}x_{33} & x_{31}x_{32}^2 & x_{32}x_{33}^2 & x_{31}x_{33}^2 & x_{31}x_{32}x_{33}\\
x_{41}x_{42}x_{43} & x_{41}^2x_{43} & x_{41}^2x_{42} & x_{42}^2x_{43} & x_{41}x_{42}x_{43} & x_{41}x_{42}^2 & x_{42}x_{43}^2 & x_{41}x_{43}^2 & x_{41}x_{42}x_{43}\\
x_{51}x_{52}x_{53} & x_{51}^2x_{53} & x_{51}^2x_{52} & x_{52}^2x_{53} & x_{51}x_{52}x_{53} & x_{51}x_{52}^2 & x_{52}x_{53}^2 & x_{51}x_{53}^2 & x_{51}x_{52}x_{53}\\
x_{61}x_{62}x_{63} & x_{61}^2x_{63} & x_{61}^2x_{62} & x_{62}^2x_{63} & x_{61}x_{62}x_{63} & x_{61}x_{62}^2 & x_{62}x_{63}^2 & x_{61}x_{63}^2 & x_{61}x_{62}x_{63}\\
x_{71}x_{72}x_{73} & x_{71}^2x_{73} & x_{71}^2x_{72} & x_{72}^2x_{73} & x_{71}x_{72}x_{73} & x_{71}x_{72}^2 & x_{72}x_{73}^2 & x_{71}x_{73}^2 & x_{71}x_{72}x_{73}\\
x_{81}x_{82}x_{83} & x_{81}^2x_{83} & x_{81}^2x_{82} & x_{82}^2x_{83} & x_{81}x_{82}x_{83} & x_{81}x_{82}^2 & x_{82}x_{83}^2 & x_{81}x_{83}^2 & x_{81}x_{82}x_{83}\\\end{bmatrix}
\end{equation}
which one can see is rank deficient because its first, fifth and ninth columns are the same. \end{proof}  

In order to prove the `{\em only-if}'' direction of Theorem~\ref{thm:k=8 main result}, we develop a number of tools in \S~\ref{subsec:trinity}. The proof of Theorem~\ref{thm:k=8 main result} will then be completed in \S~\ref{subsec:reverse direction proof}.

\subsection{The Trinity of Lines, Quadrics and Cremona Transformations}
\label{subsec:trinity}
In order to establish the \emph{trinity correspondence}, we need to introduce some genericity conditions for our main objects of interest. 
We say that a line $\ell\subset\PP(\CC^{3\times 3})$ is \textit{generic} if it contains exactly three rank-two matrices. 
These lines are generic in the usual sense, since almost all lines in $\PP(\CC^{3\times 3})$ intersect the degree-three determinantal variety $\mathcal{D} := \{ X \in \PP(\CC^{3 \times 3}) \,:\, \det(X) = 0 \}$ in three distinct points. Furthermore, given a pair of linear projections $\pi_1,\pi_2 \,:\, \PP^3 \dashrightarrow \PP^2$ with distinct centers $c_1,c_2$ we say that a smooth quadric $Q$ through $c_1,c_2$ is \textit{permissible} if it does not contain the line $\overline{c_1c_2}$ connecting the two centers. 
\newline

\begin{thm} [Trinity correspondence] \label{thm:trinity correspondence}
Consider the following three sets:
\begin{enumerate}
\item $\mathcal{L}$: the set of all generic lines $\ell$ in $\PP(\CC^{3\times 3})$, 
\item $\mathcal{Q}$: the set (up to projective equivalence) of pairs of linear projections $\pi_1,\pi_2:\PP^3 \dashrightarrow \PP^2$ with non-coincident centers $c_1,c_2$, along with a permissible quadric $Q\subset\PP^3$ through $c_1,c_2$,
\item $\mathcal{C}$: the set of (non-degenerate) Cremona transformations from 
$\PP^2 \dashrightarrow \PP^2$. 
\end{enumerate}
Then there is a $1:1$ correspondence between $\mathcal{L}$ and $\mathcal{C}$, a $1:3$ correspondence between $\mathcal{L}$ and $\mathcal{Q}$, and a $3:1$ correspondence between $\mathcal{Q}$ and 
$\mathcal{C}$, such that  diagram \eqref{eq:trinity} commutes.
\begin{equation}\label{eq:trinity}
\begin{tikzcd}[column sep=small]
&\mathcal{Q}\arrow[dr, "3\colon 1"]\\
\mathcal{L}\arrow[ur, "1\colon 3"]\arrow[rr, leftrightarrow]&&\mathcal{C}
\end{tikzcd}
\end{equation}
\end{thm}

A similar theorem holds for lines which pass through exactly two rank-two matrices; however, we do not prove it here.

We first show that for fixed linear projections $\pi_1,\pi_2$ with centers $c_1 \neq c_2 \in \PP^3$, there is a bijection between the quadrics that contain $c_1,c_2$ and lines in $\PP(\CC^{3 \times 3})$ through the fundamental matrix $F$ of $(\pi_1,\pi_2)$. This result is well-known in the context of computer vision (\cite{bratelund},\cite{hartley-kahl}), but we write an independent proof below.

\begin{lem}\label{lem:lines through F gives quadric}
Fix a pair of linear projections $\pi_1, \pi_2 \,:\, \PP^3 \dashrightarrow \PP^2$ with non-coincident centers $c_1, c_2$, and let $F$ be its fundamental matrix.
There is a $1:1$ correspondence between the quadrics $Q \subset \PP^3$ through $c_1,c_2$
and lines $\ell \subset \PP(\CC^{3 \times 3})$ 
through $F$. 
\end{lem}

\begin{proof}
 After projective transformations, we can assume that $c_1 = (1:0:0:0)$,  $c_2=(0:1:0:0)$, 
$\pi_1(u_1:u_2:u_3:u_4) = (u_2:u_3:u_4)$ and $\pi_2(u_1:u_2:u_3:u_4) = (u_1:u_3:u_4)$. 
If $F=(F_{ij})$ is the fundamental matrix of $(\pi_1,\pi_2)$, then  for all $u \in \PP^3$,
\begin{align} \label{eq:epipolar eqn}
0 = \pi_2(u)^\top F \pi_1(u) = \langle F, \pi_2(u) \pi_1(u)^\top \rangle 
= 
\left\langle  \begin{pmatrix} F_{11} & F_{12} & F_{13} \\ F_{21} & F_{22} & F_{23} \\ F_{31} & F_{32} & F_{33}\end{pmatrix}, 
\begin{pmatrix} u_1u_2 & u_1u_3 & u_1u_4 \\ u_2u_3 & u_3^2 & u_3u_4 \\ u_2u_4 & u_3u_4 & u_4^2 \end{pmatrix}
\right\rangle. 
\end{align}
Since the entries in position $(2,3)$ and $(3,2)$ of 
$\pi_2(u) \pi_1(u)^\top$ are the same,  $F$ is a scalar multiple of 
$$\begin{pmatrix} 0 & 0 & 0 \\ 0 & 0 & 1 \\ 0 & -1 & 0\end{pmatrix}$$
and $B_F(x,y) = x_3y_2 - x_2y_3$. In particular, there exists some $p\in \PP^3$  with $\pi_1(p) = x$ and 
 $\pi_2(p) = y$ if and only if $x_3y_2 = x_2y_3$.

Consider the image of $\varphi:\PP^3\dashrightarrow \PP(\CC^{3 \times 3})$  where 
$\varphi(u) = \pi_2(u) \pi_1(u)^\top$. 
By \eqref{eq:epipolar eqn}, 
$\varphi(\PP^3)$ is contained in the hyperplane $F^\perp \subset \PP(\CC^{3 \times 3})$. 
Any matrix in $\PP(\CC^{3 \times 3})$ can be written as $sF + M$ for some scalar $s$ and $M \in F^\perp$. 
Therefore, 
\begin{equation}\langle sF+M, \pi_2(u) \pi_1(u)^\top \rangle =  
  \pi_2(u)^\top M \pi_1(u) \end{equation}
since $\pi_2(u)^\top F \pi_1(u) = 0$, 
and any linear function on the image of $\varphi$ can be identified with its image in $F^\perp$.  On the other hand, 
a line $\ell$ in $\PP(\CC^{ 3 \times 3})$ through $F$
is of the form  
$\{sF + tM:(s:t)\in \PP^1\}$, where $M \in F^\perp$. Therefore, lines through $F$ are 
in bijection with linear functions on $\varphi(\PP^3)$, up to scaling. 

The monomials $\{u_1u_2, u_1u_3,u_1u_4, u_2u_3, u_2u_4,u_3^2, u_3u_4, u_4^2\}$
form a basis for the $7$-dimensional vector space of homogeneous quadratic polynomials that vanish on $c_1, c_2$. 
Therefore any quadratic polynomial in 
$\CC[u_1,u_2,u_3,u_4]_2$ vanishing at $c_1$ and $c_2$ can be written as $\langle M, \pi_2(u) \pi_1(u)^\top \rangle$ for a unique matrix $M\in F^{\perp}$.  This gives a linear isomorphism between linear functions on the image of $\varphi$, up to global scaling, (which have been identified with lines through $F$) and quadrics passing though $c_1$ and $c_2$. 
\end{proof}

\begin{cor} \label{cor:line-quadric}
Let $\pi_1,\pi_2:\PP^3\dashrightarrow \PP^2$ be two linear projections with centers $c_1 \neq c_2$ and fundamental matrix $F$. Let $\ell_F$ be a line in $\PP(\CC^{3 \times 3})$ through $F$. The correspondence $\ell_F\mapsto Q$, where $Q \subset \PP^3$ is a quadric passing through $c_1,c_2$, is as follows. 
Let $M\in\ell_F$ be any $M\neq F$. Then $Q$ is cut out by the bilinear form
\begin{equation}
B_{M}(\pi_1(p),\pi_2(p))=\pi_2(p)^\top M\pi_1(p)=0.
\end{equation}
\end{cor}

The following result is well-known and can be proven by writing a comprehensive list of the equivalence classes, under projective transformation, of quadrics through a pair of distinct points and then testing an example from each class.

\begin{lem}\label{res:generic iff smooth}
(\cite{bratelund}, \cite{hartley-kahl}, \cite[Result 22.11]{hartley-zisserman-2004}) Under the $1:1$ correspondence in Lemma~\ref{lem:lines through F gives quadric}, the line $\ell$ corresponds to a permissible quadric $Q$ through $c_1,c_2$ if and only if $\ell$ is a generic line.
\end{lem}

Next we prove that permissible quadrics through $c_1,c_2$ give rise to quadratic Cremona transformations from $\PP^2 \dashrightarrow \PP^2$. Recall that all Cremona transformations we consider are assumed to be non-degenerate.

\begin{lem}\label{lem:QFtoCremona}
Fix $\pi_i:\PP^3\dashrightarrow\PP^2$ to be linear projections with non-coincident centers $c_i$ for $i=1,2$. A permissible quadric $Q$ through $c_1,c_2$ defines a Cremona transformation $f:\PP^2\dashrightarrow\PP^2$
so that for any point $p\in Q$, $f(\pi_1(p)) = \pi_2(p)$. 
The base points of $f$ are $\pi_1(c_2)$ and the image under $\pi_1$ 
of the two lines contained in $Q$ passing through $c_1$. Similarly, the base points of $f^{-1}$ are $\pi_2(c_1)$ and the image under $\pi_2$ of the two lines contained in $Q$ passing through $c_2$.
\end{lem}

\begin{proof}
Since $c_1, c_2 \in Q$, the restriction of $\pi_1$ (and $\pi_2$) to $Q$ is generically $1:1$. Therefore, $\pi_1(Q)$ and $\pi_2(Q)$ are each birational to a $\PP^2$. 
The map $f$ will be $\pi_2\circ (\pi_1|_Q)^{-1}$. Let us check that this is a quadratic Cremona transformation. 
 
As before, we can take $\pi_1(u) = (u_2:u_3:u_4)$ and $\pi_2(u) = (u_1:u_3:u_4)$.  Then $c_1 = (1:0:0:0)$ is the kernel of $\pi_1$, and we are given that it lies on $Q$. As we saw already, these assumptions imply that $Q$ is defined by the vanishing of a polynomial of the form 
$q(u)=\alpha u_1u_2+\beta u_1+\gamma u_2+\delta$ where $\alpha\in\CC$ is a scalar, $\beta, \gamma\in\CC[u_3,u_4]$ are of degree $1$, and $\delta\in\CC[u_3,u_4]$ is of degree $2$. We can then write $q$ as 
\begin{equation}
q(u) = a u_1+ b 
\end{equation}
where $a = (\alpha u_2 + \beta)$, $b = (\gamma u_2 + \delta)\in \CC[u_2, u_3, u_4]$
with $\deg(a)=1$, $\deg(b)= 2$. The map $(\pi_1|_Q)^{-1}$ is 
then given by 
\begin{align}
x \mapsto (-b(x): x_1 a(x): x_2 a(x): x_3 a(x)) =: (u_1: u_2: u_3: u_4).
\end{align}
To verify this, first check that $\pi_1(u) = a(x) \cdot x$ where $\cdot$ denotes scalar multiplication.
To see that $u \in Q$
\begin{equation}
\begin{split}
q(u)&=u_1\cdot a(u_2,u_3,u_4)+b(u_2,u_3,u_4)\\ 
&=u_1\cdot a(\pi_1(u))+b(\pi_1(u))\\ 
&=-b(x)\cdot a(a(x)\cdot x)+b(a(x)\cdot x)\\ 
&=-b(x)a(x)a(x)+a(x)^2b(x) = 0
\end{split}
\end{equation}
where the last equality comes from the homogeneity of $a,b$
with $\deg(a)=1$, $\deg(b)= 2$. 

Composing with $\pi_2$ we have  
\begin{align}
\pi_2\circ (\pi_1|_Q)^{-1}(x) = (-b(x): x_2 a(x): x_3 a(x)), 
\end{align}
whose coordinates are indeed quadratic. Since $f = \pi_2\circ (\pi_1|_Q)^{-1}$ is defined by quadratics and generically $1:1$, it is a quadratic Cremona transformation. 

To show that this transformation is non-degenerate, we must demonstrate that it has three unique base points. To understand the base points of $f$, recall that on a smooth quadric surface there are two distinct (possibly complex) lines passing through each point. The images of the two lines passing through $c_1$ under the projection $\pi_1$ will each be a single point. Therefore $f$ is not well-defined on these image points in $\PP^2$. Similarly, $f$ is undefined on $\pi_1(c_2)$ since $\pi_2(\pi_1^{-1}(\pi_1(c_2))) = \pi_2(c_2) = 0$. Therefore these three points are exactly the base points of $f$ in the domain. Finally, because $\overline{c_1c_2}\not\subset Q$, these base points are all distinct.
The base points in the codomain can be found symmetrically. 
\end{proof}

Thus far we have shown that if we fix linear projections $\pi_1, \pi_2 \,:\, \PP^3 \dashrightarrow \PP^2$ with centers $c_1 \neq c_2 \in \PP^3$, then there is a bijection between permissible quadrics through $c_1, c_2$ and generic lines through the fundamental matrix $F$ of $(\pi_1,\pi_2)$. Furthermore, there is a map sending each generic line through $F$ (permissible quadric through $c_1,c_2$) to the  Cremona transformation from $\PP^2 \dasharrow \PP^2$ given by $\pi_2\circ(\pi_1|_Q)^{-1}$. These correspondences are summarized in \eqref{eq:F trinity}, where $\mathcal{L}_F$ is the set of all generic lines through $F$ and $\mathcal{Q}_F$ is the set of all permissible quadrics through $c_1,c_2$.

\begin{equation} \label{eq:F trinity}
\begin{tikzcd}[column sep=small]
&\mathcal{Q}_F\arrow[dr]\\
\mathcal{L}_F\arrow[ur,leftrightarrow]\arrow[rr, dashrightarrow]&&\mathcal{C}
\end{tikzcd}
\end{equation}

We can make the correspondence between generic lines through $F$ and Cremona transformations even more explicit. 

\begin{lem}\label{lem:lineToCremona}
Given a generic line $\ell \subset\PP(\CC^{3\times 3})$, the set of points $(x,y)\in \PP^2\times \PP^2$ satisfying $y^TMx=0$ for all $M\in \ell$ coincides with the closure of the graph $\{(x,f(x)):x\in \PP^2\backslash B(f)\}$ of a unique Cremona transformation $f:\PP^2\dashrightarrow \PP^2$.
This gives a $1:1$ correspondence between generic lines $\ell\subset\PP(\CC^{3\times 3})$ and Cremona transformations $f:\PP^2 \dashrightarrow \PP^2$. 
Moreover, when $F\in \ell$ has rank two, this Cremona transformation agrees with that induced by the maps
$\mathcal{L}_F \to \mathcal{Q}_F \to \mathcal{C}$.
\end{lem}
\begin{proof}
Since $\ell$ is generic, we may assume without loss of generality that $\ell = {\rm span}\{F, M\}$ where $F$ has rank two. This gives a pair of linear projections $\pi_1, \pi_2:\PP^3\dashrightarrow \PP^2$ with non-coincident centers $c_1, c_2$ 
with fundamental matrix $F$. 
In the 1:1 correspondence $\mathcal{L}_F \leftrightarrow \mathcal{Q}_F$ given in Corollary~\ref{cor:line-quadric}, the line $\ell$ corresponds to 
the permissible quadric $Q$ given by the zero set of $q(u) = \pi_2(u)^{\top}M\pi_1(u)$. 
By Lemma~\ref{lem:QFtoCremona}, the Cremona transformation $f:\PP^2\dashrightarrow \PP^2$ corresponding to $q(u)$ in the correspondence $\mathcal{Q}_F \to \mathcal{C}$ satisfies 
$f(\pi_1(p)) = \pi_2(p)$ for all $p\in Q\backslash\{c_1,c_2\}$. Since $\pi_1(Q)$ is dense in $\PP^2$, 
the graph of $f$ and the set $\{(\pi_1(p), \pi_2(p)): p\in Q\backslash \{c_1, c_2\}\}\subset \PP^2\times \PP^2$ are both two-dimensional, as is their intersection. 
Each is the image of an irreducible variety under a rational map and so the 
Zariski-closures of these two sets are equal. 
By construction, this is contained in the zero sets of 
$y^TFx$ and $y^TMx$, as $\pi_2(p)^TF\pi_1(p)=0$ for all $p\in \PP^3$ and 
$\pi_2(p)^TM\pi_1(p)=0$ for all $p\in Q$. Since $F, M$ are linearly 
independent, the variety $\{(x,y) \,:\, y^TFx = y^TMx = 0 \}$ in $\PP^2\times \PP^2$ 
is two-dimensional. It therefore coincides with the Zariski-closure of the graph of $f$.

Conversely, suppose $f:\PP^2 \dashrightarrow \PP^2$ is a Cremona transformation. 
We claim that $\{f(x)x^\top : x\in \CC^3\}$
spans a $7$-dimensional linear space $V \subset \CC^{3 \times 3}$. Up to projective transformations on $\PP_x^2$ and $\PP_y^2$, we can take $f$ to be the standard Cremona involution, giving 
\begin{equation}\label{eq:cremona 7-dim space}
f(x)x^\top = \begin{pmatrix} x_1x_2x_3 & x_1^2x_3 & x_1^2x_2\\
x_2^2x_3 & x_1x_2x_3 & x_1x_2^2\\
x_2x_3^2 & x_1x_3^2 & x_1x_2x_3 \end{pmatrix}. 
\end{equation}
One can check explicitly that seven distinct monomials appear in this matrix and so the span of all such matrices is $7$-dimensional. 
Projectively, the orthogonal complement gives a line $\ell=V^\perp$ in $\PP(\CC^{3 \times 3})$. 
By definition, $\ell$ is exactly the set of all matrices $M$ such that $y^\top Mx=0$ for all $(x,y)$ in the graph of $f$. Under the assumption that $f$ is the standard Cremona transformation, $\ell$ is the span of the diagonal matrices $ F_1 = {\rm diag}(1,-1,0)$ and $F_2 = {\rm diag}(0,1,-1)$; in general $\ell$ will be projectively equivalent to this line. We can verify that this line contains exactly the three rank-two matrices, $F_1, F_2, F_1+F_2$, and is therefore generic. 
\end{proof}
\begin{rem}\label{rem:lineToCremonaEq} 
Given $\ell = {\rm span}\{F, M\}$ we can solve for the coordinates of the corresponding Cremona transformation $f:\PP^2\dashrightarrow \PP^2$ as follows.  Given $x\in \PP^2$, the corresponding point $y=f(x)$ will be the left kernel of the $3\times 2$ matrix $\begin{pmatrix} Fx & Mx\end{pmatrix}$. The coordinates of $y$ can be written explicitly in terms of the $2\times 2$ minors of this matrix, which are quadratic in $x$. Note that, up to scaling, this formula for $y$ is independent of the choice of basis $\{F, M\}$ for $\ell$. 
Any point $x\in \PP^2$ for which $\begin{pmatrix} Fx & Mx\end{pmatrix}$ has rank $\leq 1$ will be a base point of this Cremona transformation. In particular, if $Fx=0$, then $x$ is a base point of $f$. 
As we will see below, there are three such points when ranging over all rank-two matrices in $\ell$. 
\end{rem}

The next two results finish off the proof of the trinity correspondence \eqref{eq:trinity} and proof of Theorem~\ref{thm:trinity correspondence}.

\begin{lem}\label{lem:lines have well-defined cremona transformations} 
Let  $\ell$ be a generic line in $\PP(\CC^{3\times 3})$, i.e., $\ell$ contains three rank-two matrices 
$F_1, F_2, F_3$.
\begin{enumerate}
\item Then $\ell$ gives rise to three permissible quadrics $Q_1,Q_2,Q_3\subset\PP^3$, each containing the centers of a pair of linear projections with fundamental matrices $F_1,F_2,F_3$ respectively. 
\item The quadrics $Q_1,Q_2,Q_3$, in conjunction with their distinguished linear projections, all induce the same Cremona transformation $f$. The base points of $f$ are $e_1^x,e_2^x,e_3^x$ in the domain and $e_1^y,e_2^y,e_3^y$ in the codomain, where $e_i^x$ and $e_i^y$ generate the right and left nullspaces of $F_i$ respectively.
\end{enumerate}
\end{lem}

\begin{proof}
A generic line $\ell \subset \PP(\CC^{3 \times 3})$ intersects the determinantal variety $\mathcal{D}$ cut out by $\det X = 0$  in three rank-two matrices $F_1,F_2,F_3$. 
Each $F_i$ is the fundamental matrix of a pair of linear projections $\PP^3\dashrightarrow\PP^2$ with noncoincident centers and, by Lemma~\ref{lem:lines through F gives quadric} and Lemma~\ref{res:generic iff smooth} there is a unique permissible quadric $Q_i$ through these centers corresponding to the line $\ell$.
By Lemma~\ref{lem:lineToCremona}, each of these quadrics induce the same Cremona transformation $f:\PP^2\dashrightarrow\PP^2$. 

To conclude, we show that the base points of $f$ and $f^{-1}$ are $e_1^x, e_2^x, e_3^x$ and $e_1^y, e_2^y, e_3^y$, respectively. We show that $e_1^x, e_2^x, e_3^x$ are the base points of $f$ and the argument for the base points of $f^{-1}$ follows symmetrically. 
First, note that each $e_i^x$ is a base point of $f$. This follows from Remark~\ref{rem:lineToCremonaEq}, since each $F_i\in \ell$ has rank two. 
Since the Cremona transformation $f$ has three base points, 
it only remains to show that these points are distinct.  
If $e_1^x = e_2^x$, then by linearity $Fe_1^x=0$ for all $F\in \ell = {\rm span}\{F_1, F_2\}$. 
This would imply that ${\rm rank}(F)\leq 2$ for all $F\in \ell$, contradicting genericity of the line $\ell$.
\end{proof}

\begin{cor} 
The correspondence $\mathcal{Q}\to\mathcal{C}$ is $3:1$.
\end{cor}
\begin{proof}
Let $\mathbf{Q}=(Q,\pi_1,\pi_2)\in\mathcal{Q}$ be a permissible quadric along with a pair of linear projections that correspond to $f\in\mathcal{C}$. If $F$ is the fundamental matrix associated to $(\pi_1, \pi_2)$, then there exists a unique generic line $\ell$ through $F$ corresponding to $Q$ by Lemma~\ref{lem:lines through F gives quadric} and Lemma~\ref{res:generic iff smooth}. With the full trinity correspondence, this line $\ell$ contains three fundamental matrices $F_1,F_2,F_3$ corresponding to $\mathbf{Q}_1,\mathbf{Q}_2,\mathbf{Q}_3 \in \mathcal{Q}$ that each produce the Cremona transformation $f$. Moreover, by Lemma~\ref{lem:lineToCremona}, this line $\ell$ is the unique line in $\PP(\CC^{3\times 3})$ 
corresponding to $f$. Therefore if $\mathbf{Q}'\in\mathcal{Q}$ is such that $\mathbf{Q}'\mapsto f$ it follows that $\pi'_1,\pi'_2$ have one of $F_1,F_2,F_3$ as their fundamental matrix and that the quadric $Q'$ is produced by the line $\ell$. We conclude that $\mathbf{Q}'$ is, up to projective equivalence, one of $\mathbf{Q}_1,\mathbf{Q}_2,\mathbf{Q}_3$.
\end{proof}

This completes the proof of Theorem~\ref{thm:trinity correspondence}. A consequence of Theorem~\ref{eq:trinity} is the following generalization of Problem~\ref{prob:geom}.

\begin{thm}
Given a generic codimension-two subspace $V \subset \PP(\CC^{3\times 3})$, the intersection of $V$ with $R_1$, the Segre embedding of $\PP^2\times\PP^2$, is a del Pezzo surface of degree six, and can be described explicitly via the trinity correspondence. Specifically, if $g:\PP^2\dashrightarrow\PP^2$ is the Cremona transformation corresponding to the line $V^\perp$, then  $$V\cap R_1=\{g(x)x^\top:x\in\PP^2\}\cup\{xg^{-1}(x)^\top:x\in\PP^2\}.$$
\end{thm}
\begin{proof}
For convenience, we denote
$$V_1:=\{g(x)x^\top:x\in\PP^2\}\cup\{xg^{-1}(x)^\top:x\in\PP^2\}.$$ 
To see that this is a degree-six del Pezzo surface, we show that $V_1$ can be obtained as the blow up of $\PP^2$ in three non-collinear points, specifically, at the base points of $g$: $e_1^x,e_2^x,e_3^x$. 
Let $\pi_x:V_1\dashrightarrow\PP^2$ be the morphism defined by $\pi_x(vu^\top)=u$. Let $\ell_i^y$ be the exceptional lines of $g$ such that $g^{-1}(\ell_i^y)=e_i^x$. Then $\pi_x$ is $1:1$ except on three mutually skew lines $\{y(e_i^x)^\top:y\in\ell_i^y\}$ which are taken to the points $\{e_i^x\}$. Therefore $V_1$ is the blow up of $\PP^2$ in three non-collinear points and is a del Pezzo surface of degree six. 

In particular, $V_1$ must be Zariski closed and it follows by Lemma \ref{lem:lineToCremona} that $V\cap R_1=V_1$.
\end{proof}

\subsection{Back to the proof of Theorem~\ref{thm:k=8 main result}}
\label{subsec:reverse direction proof}
Before we can adapt the trinity correspondence to the reconstruction of point pairs, we need to address a certain kind of degeneracy. Given a configuration of point pairs $P = (x_i,y_i)_{i=1}^k$ consider the matrix $Z=(x_i^\top \otimes y_i^\top)_{i=1}^k$ and its right nullspace $\mathcal{N}_Z$.

\begin{lem}\label{lem:P-degen}
Suppose $P=(x_i,y_i)_{i=1}^k$ admits a generic line $\ell \subseteq \mathcal{N}_Z$  (passing through three rank-two matrices $F_1,F_2,F_3$). Then for all $j=1,2,3,$ there is no $i$ such that $y_i^\top F_j=0 = F_jx_i$.
\end{lem}

\begin{proof}
Suppose, without loss of generality, $y_1^\top F_1=0 = F_1x_1$.
From the matrix $F_1$ and the line $\ell$ through it we obtain a pair of projections $\pi_1,\pi_2$ with centers $c_1,c_2$, and a smooth permissible quadric $Q$ passing through them. Then 
$\pi_2(c_1)$ and $\pi_1(c_2)$ are the left and right epipoles of $F_1$, but since 
$y_1^\top F_1=0 = F_1x_1$, it must be that $y_1 \sim \pi_2(c_1)$ and $x_1 \sim \pi_1(c_2)$. On the other hand, for any point $p$ on the line connecting $c_1,c_2$, 
$$\pi_2(p)^\top F_2\pi_1(p)= \pi_2(c_1) F_2 \pi_1(c_2) = y_1^\top F_2x_1=0$$ 
since $F_2 \in \mathcal{N}_Z$. Therefore, by Corollary~\ref{cor:line-quadric}, 
$p\in Q$ and thus $\overline{c_1c_2}\subset Q$, which is a contradiction since $Q$ is permissible. 
\end{proof}

Even though a rank-two matrix $F$ on a generic line in $\mathcal{N}_Z$ cannot have 
$y_i^\top F = 0 = F x_i$, it might be that one of the equations hold. We  
name this type of degeneracy in the following definition. 

\begin{defn}
A generic line $\ell \subseteq \mathcal{N}_Z$ is \textbf{$P$-degenerate} if there exists a rank-two matrix $F\in\ell$ such that either $Fx_i=0$ or $y_i^\top F=0$ for some $i$. 
We call a generic line that is not $P$-degenerate a \textbf{$P$-generic} line. 
\end{defn}

Any rank-two matrix $F$ in a $P$-generic line will give a reconstruction 
$c_1, c_2, p_1, \hdots, p_k$ of the point pairs $P$. That is, there will be 
linear projections $\pi_1, \pi_2:\PP^3\dashrightarrow\PP^2$ with centers $c_1, c_2$
so that $\pi_2^{\top}(p)F\pi_1(p)=0$ for all $p\in \PP^3$ and $(x_i, y_i) = (\pi_1(p_i), \pi_2(p_i))$ for all $i=1, \hdots, k$. 
A smooth quadric $Q$ will contain two lines through any of its points. 

\begin{defn}
A quadric $Q\subset \PP^3$ passes \textbf{degenerately} through a reconstruction $c_1, c_2$, $\{p_i\}_{i=1}^k$ of $P$ if it passes through these $k+2$ points and contains the line through a center point $c_i$ and reconstructed point $p_j$.
\end{defn}

\begin{defn}
A Cremona transformation $f:\PP^2\dashrightarrow \PP^2$ maps $x_i\mapsto y_i$ \textbf{degenerately} if $x_i$ is a base point of $f$ and $y_i$ lies on the corresponding exceptional line, or symmetrically,  $y_i$ is a base point of $f^{-1}$ and $x_i$ lies on the corresponding exceptional line. 
\end{defn}

Generically, the trinity correspondence specializes to the reconstruction of point pairs in an intuitive way.

\begin{thm}\label{thm:trinity correspondence (vision)}
Given a configuration of point pairs $P = (x_i,y_i)_{i=1}^k$ and 
the matrix $Z=(x_i^\top \otimes y_i^\top)_{i=1}^k$, define the  
following subsets of $\mathcal{L},\mathcal{Q},\mathcal{C}$: 
\begin{enumerate}
    \item $\mathcal{L}_{P}$: the set of all $P$-generic lines $\ell\subseteq \mathcal{N}_{Z}:= \textup{nullspace}(Z)$, 
    \item $\mathcal{Q}_P$: the set (up to projective equivalence) of all permissible quadrics passing non-degenerately through some reconstruction $c_1,c_2, p_1, \hdots, p_k$ of $P$, 
    \item $\mathcal{C}_{P}$: the set of all Cremona transformations $f:\PP^2\dashrightarrow\PP^2$ mapping  $x_i\mapsto y_i$ non-degenerately for all $i =1,\hdots, k$.
\end{enumerate}
Then there is a $1:1$ correspondence between the elements of $\mathcal{L}_{P}$ and $\mathcal{C}_P$, 
a $1:3$ correspondence between the elements of $\mathcal{L}_{P}$ and $\mathcal{Q}_P$, and a $3:1$ 
correspondence between the elements of $\mathcal{Q}_P$ and $\mathcal{C}_P$ as in:
\begin{equation}\label{eq:trinity (vision)}
\begin{tikzcd}[column sep=small]
&\mathcal{Q}_P\arrow[dr, rightarrow, "3\colon 1"]\\
\mathcal{L}_P\arrow[ur, rightarrow, "1\colon 3"]\arrow[rr, leftrightarrow]&&\mathcal{C}_P
\end{tikzcd}
\end{equation}
\end{thm}

\begin{proof}
We need to show that the trinity correspondence \eqref{eq:trinity} can be 
restricted to the sets $\mathcal{L}_P, \mathcal{Q}_P, \mathcal{C}_P$.
We will therefore examine each leg of this diagram.

($\mathcal{L}_P\to\mathcal{Q}_P$)  We begin by considering a $P$-generic line $\ell = {\rm span}\{F,M\} \subseteq \mathcal{N}_Z$.  
Without loss of generality, we can take $F$ to be one of the three fundamental matrices in $\ell$ with corresponding 
projections $\pi_1, \pi_2:\PP^3\dashrightarrow\PP^2$ with non-coincident centers $c_1, c_2$ that give reconstructions $p_1, \hdots, p_k\in \PP^3$ of the point pairs $P$. 
By Lemma~\ref{lem:lines through F gives quadric}, the line $\ell$ corresponds to a smooth permissible quadric $Q$ defined by the vanishing of $q(u) = \pi_2(u)^TM\pi_1(u)$. 
For any point $p_i$ in the reconstruction, we have 
\begin{equation}\label{eq:lineToQuadricPoints}
    q(p_i) = \pi_2(p_i)^\top M\pi_1(p_i)=y_i^\top Mx_i=0
\end{equation}
since $M\in \ell\subset \mathcal{N}_Z$. Therefore $Q$ passes through the reconstruction $c_1, c_2, p_1, \hdots, p_k$. 
It remains to show that it does so non-degenerately.  
By Lemmas~\ref{lem:QFtoCremona} and~\ref{lem:lines have well-defined cremona transformations}, 
a reconstructed point $p_i$ lies on one of the lines through $c_1$ (or symmetrically through $c_2$) if and only if there exists $M\in\ell$ such that $Mx_i=0$ (or symmetrically $y_i^\top M=0)$. 
Since $\ell$ is $P$-generic there is no such $M$, implying that the quadric passes through the reconstruction non-degenerately.

($\mathcal{Q}_P\to\mathcal{C}_P$)
Consider a permissible quadric $Q$ passing through a reconstruction $c_1,c_2, p_1, \hdots, p_k$ of $P$ with linear projections $\pi_1,\pi_2$. As in Theorem~\ref{thm:trinity correspondence}, 
the tuple $(Q,\pi_1, \pi_2)$ induces a Cremona transformation $f := \pi_2 \circ (\pi_1|_{Q})^{-1}$. By Lemma~\ref{lem:QFtoCremona}, the base points of $f$
are the images of the point $c_2$ and each of the lines in $Q$ passing through $c_1$. Since $p_i\neq c_2$ and does not belong to these lines, the point $x_i = \pi_1(p_i)$ is not a base point of $f$.  Similarly, the base points of $f^{-1}$ are the images of the point $c_1$ and the lines in $Q$ passing through $c_2$ under $\pi_2$, so a symmetric argument 
shows that $y_i = \pi_2(p_i)$ is not a base point of $f^{-1}$. 
Therefore $f$ maps $x_i = \pi_1(p_i)$ to $y_i = \pi_2(p_i)$ non-degenerately. 

($\mathcal{C}_P\to \mathcal{L}_P$) Consider a Cremona transformation $f:\PP^2\dashrightarrow\PP^2$ such that $x_i\mapsto y_i$ non-degenerately for all $i$. 
As in Lemma~\ref{lem:lineToCremona}, $f$ corresponds to a unique line 
$\ell \subset \PP(\CC^{3\times 3})$ defined by the property that $f(x)^\top Mx=0$ for all $M\in\ell$ and $x\in\PP^2$. In particular, $y_i^\top Mx_i=0$ for all $M\in\ell$ and $i=1,\ldots,k$, 
implying that $\ell\subseteq \mathcal{N}_Z$. By assumption, no point $x_i$ is a base point of $f$ and no point $y_i$ is a base point of $f^{-1}$. By Lemma~\ref{lem:lines have well-defined cremona transformations}, 
it then follows that $Mx_i\neq 0$ and $y_i^{\top}M\neq 0$ for all $M\in \ell$. 
Therefore $\ell$ is not $P$-degenerate. 
\end{proof}

\begin{rem}
 The assumptions of non-degeneracy can be removed from the 1:1 correspondence between generic lines in $\mathcal{N}_Z$ and Cremona transformations mapping $x_i\mapsto y_i$. Extending this to quadrics is more subtle, as 
 some rank-two matrices  $F\in \ell\subset \mathcal{N}_Z$ may not give full reconstructions of the point pairs $P$. 
\end{rem}

\begin{proof}[Proof of Theorem~\ref{thm:k=8 main result}($\Rightarrow$)]
For $k=8$ semi-generic point pairs, the matrix $Z=(x_i^\top \otimes y_i^\top)_{i=1}^8$ is rank deficient exactly when $\mathcal{N}_Z =: \ell$ is a line. This line $\ell$ is generic because it is also the nullspace of any submatrix of $Z$ of size $7 \times 9$ and the corresponding seven point pairs are generic.  Pick a subset of seven point pairs, say $(x_i,y_i)_{i=1}^7$, from the original eight pairs. Since these seven point pairs are generic, and $\ell$ is also generic, we can assume that 
$Fx_i \neq 0$ and $y_i^\top F \neq 0$ for any rank-two matrix $F \in \ell$ and  all $i=1,\ldots,7$. On the other hand, if we pick a different set of seven 
point pairs, say $(x_i,y_i)_{i=2}^8$, then $\ell$ is also the nullspace of the 
corresponding $Z_7$ and by the same argument as before, $Fx_i \neq 0$ and $y_i^\top F \neq 0$ for any rank-two matrix $F \in \ell$ and  all $i=2,\ldots,8$. Therefore, $\ell$ is $P$-generic.

Since $\ell$ is $P$-generic, by 
Theorem~\ref{thm:trinity correspondence (vision)}, 
$\ell$ gives rise to a Cremona transformation $f:\PP^2_x\dashrightarrow\PP^2_y$ such that $f(x_i)=y_i$ for all $i=1,\ldots,8$. This finishes the proof of Theorem~\ref{thm:k=8 main result}.
\end{proof}  

We end this section by demonstrating the trinity correspondence on an example, beginning with a single quadric through a reconstruction.

\begin{ex}
Consider the quadric $Q\subset\PP^3$ defined by the equation $x^2+y^2-z^2-w^2=0$ and the following $10$ points $p_1,\ldots,p_8,c_1,c_2\in Q$:
\begin{align*}
c_1&=(1:0:0:1)\quad\quad &&c_2=(0:1:0:1)\\
p_1&=(5:12:13:0)\quad\quad &&p_2=(13:0:5:12)\\
p_3&=(12:5:13:0)\quad\quad &&p_4=(3:4:5:0)\\
p_5&=(4:3:5:0)\quad\quad &&p_6=(3:4:0:5)\\
p_7&=(4:3:0:5)\quad\quad &&p_8=(5:0:4:3)
\end{align*} 
The two projections (cameras) with centers $c_1,c_2$ have matrices
\[
A_1=\begin{bmatrix}1 & 0 & 0 & -1\\ 0 & 1 & 0 & 0\\ 0 & 0 & 1 & 0\end{bmatrix},\quad\quad 
A_2=\begin{bmatrix}1 & 0 & 0 & 0\\ 0 & 1 & 0 & -1\\ 0 & 0 & 1 & 0\end{bmatrix}
\]
and we can calculate the image points and epipoles
\begin{align*}
e_x&=(-1:1:0)\quad\quad &&e_y=(1:-1:0)\\
x_1&=(5:12:13)\quad\quad &&y_1=(5:12:13)\\
x_2&=(1:0:5)\quad\quad &&y_2=(13:-12:5)\\
x_3&=(12:5:13)\quad\quad &&y_3=(12:5:13)\\
x_4&=(3:4:5)\quad\quad &&y_4=(3:4:5)\\
x_5&=(4:3:5)\quad\quad &&y_5=(4:3:5)\\
x_6&=(-2:4:0)\quad\quad &&y_6=(3:-1:0)\\
x_7&=(-1:3:0)\quad\quad &&y_7=(4:-2:0)\\
x_8&=(2:0:4)\quad\quad &&y_8=(5:-3:4).
\end{align*}
The point pairs $(x_i,y_i)$  give us the matrix
\begin{align*}
Z_8=
\begin{bmatrix}
 25 & 60 & 65 & 60 & 144 & 156 & 65 & 156 & 169 \\
 13 & -12 & 5 & 0 & 0 & 0 & 65 & -60 & 25 \\
 144 & 60 & 156 & 60 & 25 & 65 & 156 & 65 & 169 \\
 9 & 12 & 15 & 12 & 16 & 20 & 15 & 20 & 25 \\
 16 & 12 & 20 & 12 & 9 & 15 & 20 & 15 & 25 \\
 -6 & 2 & 0 & 12 & -4 & 0 & 0 & 0 & 0 \\
 -4 & 2 & 0 & 12 & -6 & 0 & 0 & 0 & 0 \\
 10 & -6 & 8 & 0 & 0 & 0 & 20 & -12 & 16 \\
\end{bmatrix}
\end{align*}
which we can check is rank deficient and has nullspace spanned by the vectors
\begin{align*}
m_1=(-1, 1, 0, -1, -1, 0, 0, 0, 1),\quad\quad m_2=(0, 0, -1, 0, 0, -1, 1, 1, 0).
\end{align*}
The reconstruction we started with has fundamental matrix
\[
F=\begin{bmatrix} 0 & 0 & 1\\ 0 & 0 & 1\\ -1 & -1 & 0\end{bmatrix}
\]
and if we take a different matrix
\[
M=\begin{bmatrix} -1 & -1 & 0\\ 1 & -1 & 0\\ 0 & 0 & 1\end{bmatrix}
\]
in the nullspace of $Z_8$ we can verify that $A_2^\top MA_1$ yields the original quadric $Q$
\[
(x,y,z,w)A_2^\top MA_1(x,y,z,w)^\top=(x,y,z,w)\begin{bmatrix}
 -1 & -1 & 0 & 1 \\
 1 & -1 & 0 & 1 \\
 0 & 0 & 1 & 0 \\
 1 & 1 & 0 & -1 \\
\end{bmatrix}(x,y,z,w)^\top=-x^2-y^2+z^2+w^2.
\]
The other two possible choices for fundamental matrices in the nullspace of $Z_8$ are
\[
F_2=\begin{bmatrix} -1 & -1 & 1\\ 1 & -1 & 1\\ -1 & -1 & 1\end{bmatrix}\quad \text{ and }\quad
F_3=\begin{bmatrix} -1 & -1 & -1\\ 1 & -1 & -1\\ 1 & 1 & 1\end{bmatrix},
\]
which have epipoles
\begin{align*}
e_x^2&=(0:1:1)\quad\quad &e_y^2=(-1:0:1)\\
e_x^3&=(0:-1:1)\quad\quad &e_y^3=(1:0:1).
\end{align*}
Moreover, we can verify that there is a unique Cremona transformation
\[
f(x_1,x_2,x_3)=(x_1^2-x_2^2+x_3^2,x_1^2+2 x_1 x_2+x_2^2-x_3^2,2 x_1 x_3)
\]
such that $f(x_i)=y_i$ for all $i$. This Cremona transformation has base points exactly matching the epipoles. Finally, we can check that each camera center lies on two real lines on the quadric $Q$, parameterized by $(a:b)\in \PP^1$ as
\begin{align*}
\ell_x^2=(a:b:b:a),\quad\ell_x^3=(a:-b:b:a),\quad \ell_y^2=(-b:a:b:a), \quad \text{ and } \quad\ell_y^3=(b:a:b:a)
\end{align*}
whose images are exactly the other two possible pairs of epipoles/base points $(e_x^2,e_y^2)$ and $(e_x^3,e_y^3)$.
\end{ex}

\section{$k=7$}\label{sec:k=7}
We now come to the case of $k=7$ point pairs. In order to understand the case of seven point pairs, 
we will first need to understand six generic point pairs $(x_i,y_i)_{i=1}^6$. In this case, the nullspace $\mathcal{N}_Z$ of the matrix $Z=(x_i^\top \otimes y_i^\top)_{i=1}^6$ is projectively a plane and $\mathcal{N}_Z \cap \mathcal{D} =:C $ is a cubic curve in $\PP(\CC^{3\times 3})$ lying in the plane $\mathcal{N}_Z$. By our genericity assumption, $C$ misses all rank-one matrices in $\mathcal{D}$ and hence every point on $C$ is a  
fundamental matrix of $(x_i,y_i)_{i=1}^6$.  
Let $\kappa_x$ and $\kappa_y$ denote the quadratic maps that 
take a rank-two matrix $M \in \PP(\CC^{3 \times 3})$ to its right and left nullvectors respectively. As a consequence of the classical theory of blowups and cubic surfaces as discussed in \cite{rankdrop1}, 
the maps $C \rightarrow \kappa_x(C) =: C_x \subset \PP^2_x$ and $C \rightarrow  \kappa_y(C) =: C_y \subset \PP^2_y$ are isomorphisms when $(x_i,y_i)_{i=1}^6$ is generic; we will go into more detail on the nature of these isomorphism in Section 4.2.1.  
\begin{equation} \label{eq:adjoint maps}
\begin{tikzcd}[column sep=small]
&C \arrow[dr, "\kappa_y"]\\
\PP^2_x \supset C_x\arrow[ur,leftarrow, "\kappa_x"]&&C_y \subset \PP^2_y
\end{tikzcd}
\end{equation}
By the composition $\kappa_y \circ \kappa_x^{-1}$, we get that $C_x$ and $C_y$ are isomorphic cubic curves. However, this isomorphism is not particularly useful; for instance, it does not take $x_i\mapsto y_i$.
By construction, the curves $C_x$ and $C_y$ consist exactly of all 
possible epipoles of the fundamental matrices of $(x_i,y_i)_{i=1}^6$ in $\PP^2_x$ and $\PP^2_y$. We therefore call $C_x$ and $C_y$ the right and left {\em epipolar curves} of 
$(x_i,y_i)_{i=1}^6$. We will see that these cubic curves are closely tied to both rank drop and the trinity relationship established in Theorem~\ref{thm:trinity correspondence (vision)}.

\begin{ex}\label{ex: k=7 example 1}
Consider the following six point pairs:
\begin{align*}
x_1=(0:0:1) \quad &y_1=(0:0:1) \quad\quad &x_2=(1:0:1)\quad &y_2=(1:0:1)\\
x_3=(0:1:1) \quad &y_3=(0:1:1) \quad\quad &x_4 =(1:1:1) \quad &y_4=(1:1:1)\\
x_5=(3:5:1) \quad &y_5=(7:-2:1) \quad\quad &x_6=(-7:11:1) \quad &y_6=(3:13:1)
\end{align*}

\begin{figure}
\fbox{%
\includegraphics[scale=0.134]{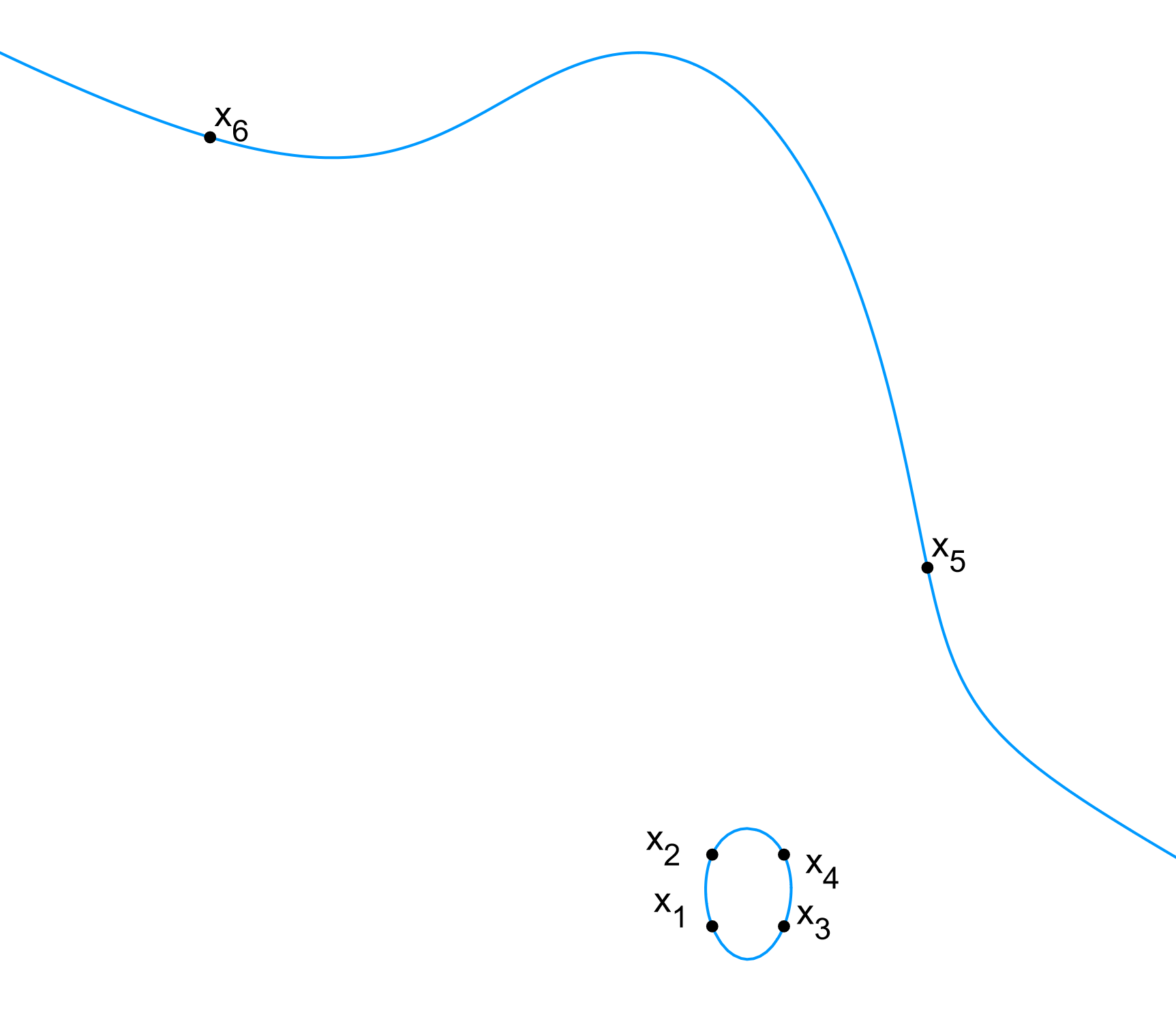}
}
\fbox{%
\includegraphics[scale=0.1382]{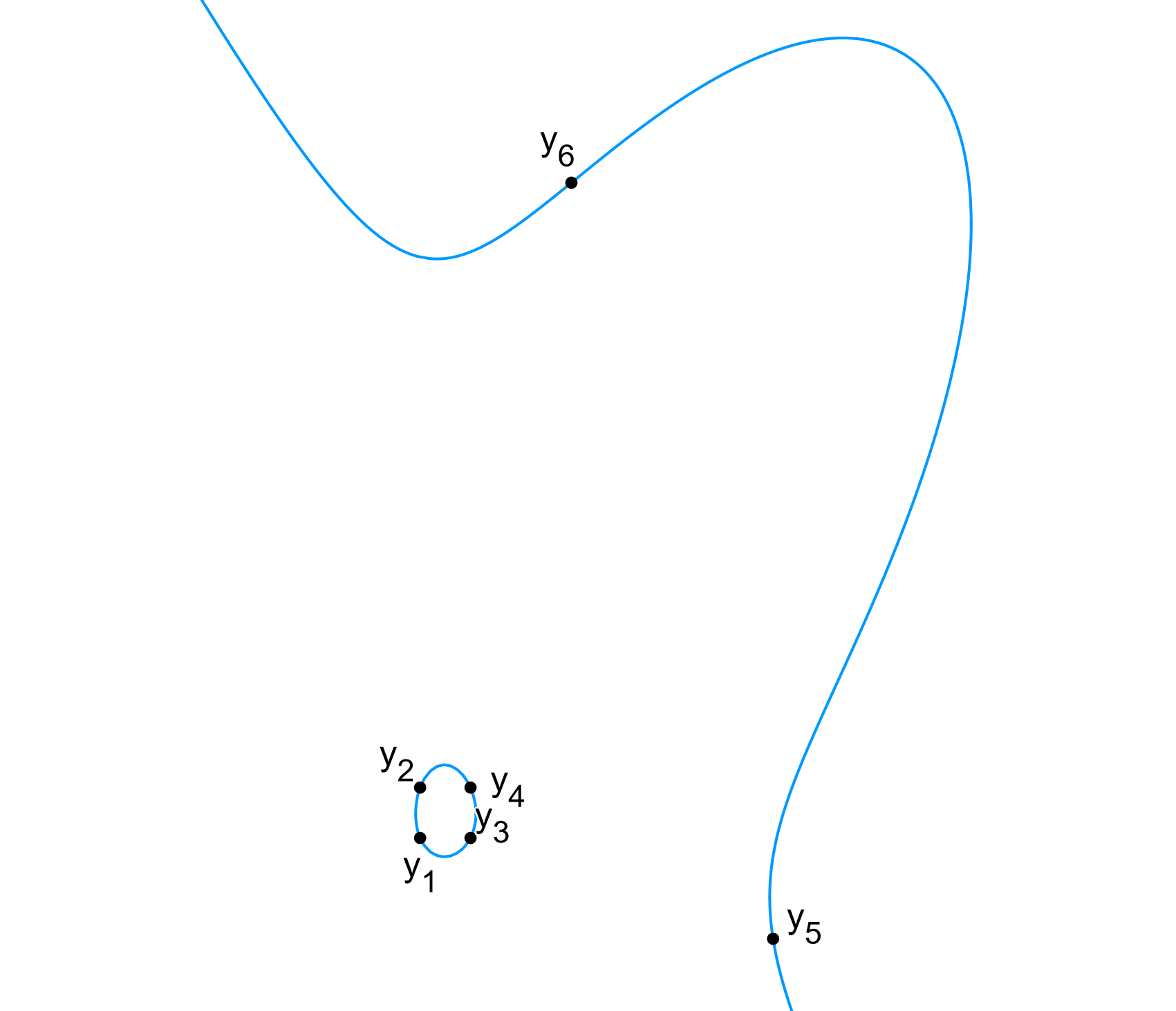}
}
\caption{The cubic curves $C_x$ and $C_y$ from Example~\ref{ex: k=7 example 1}, with $x_i$ and $y_i$ labeled. \label{fig:cubic curves}}
\end{figure}
Figure~\ref{fig:cubic curves} shows the curves $C_x$ and $C_y$. Observe that $x_i\in C_x$ and $y_i\in C_y$ for all $i=1,\ldots,6$, a fact we will prove in Section~\ref{subsec:rank drop and cubic curves}. 
The curves $C_x$ and $C_y$ are cut out by $g_x(u) =0$ and $g_y(v)=0$ in $\PP^2_x$ and $\PP^2_y$ where

$$\begin{array}{ll}
g_x(u) = & 447u_1^3+775u_1^2u_2+113u_1u_2^2+118u_2^3-4083u_1^2u_3-888u_1u_2u_3\\&-1521u_2^2u_3+3636u_1u_3^2+1403u_2u_3^2, \textup{ and }\\
& \\
g_y(v) = & 447v_1^3-136v_1^2v_2-12v_1v_2^2+118v_2^3-3608v_1^2v_3+148v_1v_2v_3\\
& -1478v_2^2v_3+3161v_1v_3^2+1360v_2v_3^2.
\end{array}$$

In Section~\ref{subsec:Cremona hexahedral} we use classical invariant theory to 
derive the polynomials $g_x$ and $g_y$. 
\qed
\end{ex}

Given seven point pairs $(x_i,y_i)_{i=1}^7$, denote the epipolar  curves obtained by excluding the $i$th point pair as $C_x^{\hat i}$ and $C_y^{\hat i}$. In the event that these curves are equal for all choices of $i$, we denote 
$C_x:=C_x^{\hat 1}=\ldots=C_x^{\hat 7}$ and $C_y:=C_y^{\hat 1}=\ldots=C_y^{\hat 7}$. We will see that this equality is 
necessary (Theorem~\ref{thm:k=7 simple}) and sufficient (Theorem~\ref{thm: k=7 central theorem}) for $Z_7 = (x_i^\top \otimes y_i^\top)_{i=1}^7$ to be rank deficient.

The maps $\kappa_x, \kappa_y$ are not the only way to derive the epipolar curves $C_x, C_y$;  it is also possible to obtain them 
via the trinity correspondence \eqref{eq:trinity (vision)}. This will be the subject of Section 4.1 and will allow us to prove the following result:

\begin{thm}\label{thm:k=7 simple}
For $(x_i,y_i)_{i=1}^7$ semi-generic point pairs, $Z_7$ is rank deficient if and only if there exists cubic curves $C_1$ through $x_1,\ldots,x_7$ and $C_2$ through $y_1,\ldots,y_7$ as well as an isomorphism $f:C_1\to C_2$ such that $x_i\mapsto y_i$. Moreover, if this holds then $C_1=C_x$ and $C_2=C_y$.
\end{thm}
This is the first of the two main results in this section and is the more geometric theorem; we will prove it at the end of Section 4.1. In Section 4.2.1 we use the theory of cubic surfaces as in \cite{rankdrop1} to obtain explicit equations for the epipolar curves. In Section 4.2.2 we use these explicit equations to characterize rank deficiency of $Z_7$ using $14$ algebraic equations and prove our second main result, Theorem~\ref{thm: k=7 central theorem}, which is the more algebraic theorem. Finally, in Section 4.3 we collect some further results outside the assumption of semi-genericity. 

\subsection{Rank Drop and Cubic Curves} \label{subsec:rank drop and cubic curves}
Before addressing the case of six generic point pairs and seven semi-generic point pairs, 
we establish an analogue of Lemma~\ref{lem:lineToCremona}
to show how general projective planes in $ \PP(\CC^{3\times 3})$ give rise to Cremona transformations of cubic curves.

\begin{lem}\label{lem:PlaneToCremonaCubic}
Let $\mathcal{P} \subset \PP(\CC^{3\times 3})$ be a projective plane not containing any rank-one matrix. 
The set of points $(x,y)\in \PP^2\times \PP^2$ satisfying $y^TMx=0$ for all $M\in \mathcal{P}$ coincides with the 
 closure of the graph $\{(x,f(x)):x\in C_x^{\mathcal{P}}\}$ of the restriction of a Cremona transformation $f:\PP^2\dashrightarrow \PP^2$ 
to a cubic curve $C_x^{\mathcal{P}}$.  Moreover there is a two-dimensional family of Cremona transformations 
$f_{\ell}:\PP^2\dashrightarrow \PP^2$, indexed by generic lines $\ell \subset \mathcal{P}$ as in Lemma~\ref{lem:lineToCremona}, with the same restriction to $C_x^{\mathcal{P}}$. 
\end{lem} 
\begin{proof}
The curve $C_x^{\mathcal{P}}$ consists of the set of points $x \in \PP^2$ for which there exists an $M\in \mathcal{P}$ with $Mx = 0$. 
When $\mathcal{P} = \mathcal{N}_Z$, this is the epipolar curve $C_x$ described above. By choosing a basis $\{M_1, M_2, M_3\}$ for $\mathcal{P}$ we can write any $M\in \mathcal{P}$ as $aM_1+bM_2+cM_3$. 
Given $x\in\PP^2$ there exists $(a:b:c) \in \PP^2$ with $(aM_1+bM_2+cM_3)x=0$ if and only if  $\det\begin{pmatrix} M_1x & M_2x & M_3x \end{pmatrix}=0$.  Therefore $C_x^{\mathcal{P}}$ is defined by the vanishing of this determinant, which is a cubic form in $x_1, x_2, x_3$. Symmetrically the cubic curve $C_y^{\mathcal{P}}$ defined by the vanishing of the determinant of the matrix with rows $y^\top M_j$  
coincides with $C_y$ when $\mathcal{P} = \mathcal{N}_Z$.

Let $\ell = {\rm span}\{M_1,M_2\} \subset \mathcal{P} \subset \PP(\CC^{3\times 3})$ be a generic line. 
By Lemma~\ref{lem:lineToCremona}, there is a Cremona transformation $f_{\ell}:\PP^2\dashrightarrow \PP^2$ whose graph is the set of 
points $(x,y)\in \PP^2\times \PP^2$ satisfying $y^TMx=0$ for all $M\in \ell$.  
As in Remark~\ref{rem:lineToCremonaEq}, the map $f_{\ell}$ is given by $x\mapstochar\dashrightarrow \ker\begin{pmatrix} M_1x & M_2x\end{pmatrix}$. 
For $x\in C_x^{\mathcal{P}}$ except the three base points of $f_{\ell}$, the left kernel of $\begin{pmatrix} M_1x & M_2x\end{pmatrix}$ is also the left kernel of the rank-two $3\times 3$ matrix $\begin{pmatrix} M_1x & M_2x & M_3x\end{pmatrix}$, which is independent of the choice of $\ell = {\rm span}\{M_1, M_2\}\subset \mathcal{P}$. 

Note that the graph $\{(x,f_{\ell}(x)):x\in C_x^{\mathcal{P}}\}$ and the set of points $(x,y)\in \PP_x^2\times \PP_y^2$ satisfying $y^TMx=0$ for all $M\in \mathcal{P}$ have the same projection onto $\PP_x^2$, namely $C_x^{\mathcal{P}}$. 
For any $x\in C_x^{\mathcal{P}}$, the corresponding point $y$ is given by $f_{\ell}(x) = \ker\begin{pmatrix} M_1x & M_2x & M_3x\end{pmatrix}$. 
 \end{proof}

\subsubsection{Six point pairs}
Let $(x_i,y_i)_{i=1}^6$ be a set of six generic point pairs, $Z= (x_i,y_i)_{i=1}^6$ and $F$ be any choice of fundamental matrix (i.e., a rank-two matrix on the projective plane $\mathcal{N}_{Z}$). Genericity 
guarantees a reconstruction $p_1,\ldots,p_6,c_1,c_2 \in \PP^3$, of $(x_i,y_i)_{i=1}^6$ from $F$. Recall that $c_1, c_2$ are the centers of camera projections $\pi_1, \pi_2$ and $p_1, \ldots, p_6$ are world points such that $\pi_1(p_j) = x_j$ and $\pi_2(p_j) = y_j$ for all $j=1, \ldots, 6$. 

Since $\mathcal{N}_Z$ is a two-dimensional plane, it contains a pencil of lines through $F$ (see \eqref{eq:F trinity} and \eqref{eq:trinity (vision)}) which corresponds to a pencil of quadrics $Q_\lambda$, each passing through the reconstruction. The intersection of these quadrics, also obtainable as the intersection of any two distinct quadrics in the pencil, is a quartic space curve $W \subset \PP^3$ that must also pass through the reconstruction. Since 
$c_1,c_2$ are on $W$, $\pi_1(W)\subset\PP^2_x$ and $\pi_2(W)\subset\PP^2_y$ are cubic curves. We will see that these 
cubic curves are independent of the choice of $F$, and that they are exactly the epipolar curves $C_x$ and $C_y$. We will use this derivation to study their special properties arising from the trinity relationship. The following lemma assumes the set up just described. 

\begin{lem}\label{lem:k=7 cubic curves characteristics}For six generic point pairs $(x_i,y_i)_{i=1}^6$, we have the following:
\begin{enumerate} 
\item The cubic curves $\pi_1(W)$ and $\pi_2(W)$ are the right and left epipolar curves $C_x,C_y$, respectively; In particular, they are independent of the choice of $F$; 
\item The points $x_i$ lie on $C_x$ and $y_i$ lie on $C_y$ for $i=1,\ldots,6$.
\item There exists a two-parameter family of Cremona transformations $f_{\ell}:\PP^2_x\dashrightarrow\PP^2_y$, indexed by  lines $\ell$ in the projective plane $\mathcal{N}_Z$, such that the following holds: 
\begin{itemize}
    \item $f_{\ell}(x_i)= y_i$ for all $i=1,\ldots,6$,
    \item the restriction of $f_{\ell}$ to a map $C_x\to C_y$ is independent of $\ell$, and
    \item the base points of all the Cremona transformations $f_\ell$ lie in $C_x,C_y$.
\end{itemize}
\end{enumerate}
\end{lem}

\begin{proof}
Let $F$ be a fundamental matrix in $\mathcal{N}_Z$.  Since $(x_i,y_i)_{i=1}^6$ is generic, $F$ can be any element of the cubic curve $C = \mathcal{N}_Z\cap\mathcal{D}$, and we can use $F$ to obtain a reconstruction consisting of world points $p_1, \hdots, p_6$ and cameras corresponding to linear projections $\pi_1, \pi_2: \PP^3\dashrightarrow \PP^2$.

The quartic space curve $W$ is defined by quadrics of the form 
$q(u) = \pi_2(u)^\top M\pi_1(u)$ where $M\in \mathcal{P}\cap F^{\perp}$.
Therefore $\pi_1(W)$ contains the cubic plane curve $C_x$ defined by $\{x\in \PP^2: \exists M\in \mathcal{N}_Z \text{ s.t. } Mx=0\}$. 
Since $c_1\in W$, $\pi_1(W)$ is a cubic plane curve and so these must be equal. 
A symmetric argument shows that $\pi_2(W) = C_y$. 
Since $W$ contains each point $p_i$, this also implies that $x_i = \pi_1(p_i)$ belongs to $C_x$ and $y_i = \pi_2(p_i)$ belongs to $C_y$
for all $i=1, \hdots, 6$. 

By Lemma~\ref{lem:PlaneToCremonaCubic}, for any generic line $\ell\subset \mathcal{N}_Z$, the restriction of the Cremona transformation 
$f_{\ell}:\PP^2\dashrightarrow \PP^2$ to the cubic $C_x$ is independent of the choice of $\ell$. 
By Theorem~\ref{thm:trinity correspondence (vision)}, $f_{\ell}(x_i) = y_i$ for all $i$. 
As in Lemma~\ref{lem:lines have well-defined cremona transformations}, the base points of $f_{\ell}$ are the right kernels of the three 
rank-two matrices $F_1, F_2, F_3\in \ell$ and therefore belong to $C_x$. Similarly, the base points of $f_{\ell}^{-1}$ 
are the left kernels of these matrices and so belong to $C_y$. 
\end{proof}

\begin{rem}
Given a rank two matrix $F\in\mathcal{N}_Z$, it may be the case that $Fx_i=0$ (or $y_i^\top F=0)$ for some $i$. However, even in this case we can still apply the trinity \eqref{eq:trinity} to obtain a pencil of quadrics (and a pencil of Cremona transformations), and from them the cubic curves $C_x,C_y$ with the isomorphism between them. Therefore, even if $\ell$ is such that $x_i$ is a base point of $f_\ell$, the restriction of $f_\ell$ to a map $C_x\to C_y$, as in Lemma~\ref{lem:Cremona transformations induce isomorphisms on cubics}, would still satisfy $x_i\mapsto y_i$
\end{rem}

\subsubsection{From six points to seven}
The trinity correspondence has allowed us to prove a number of properties of the epipolar curves corresponding to six generic point pairs. In particular, we know that there is an isomorphism $f:C_x\to C_y$ that sends $x_i\mapsto y_i$ for all $i=1,\ldots,6$, which is induced by a two-parameter family of Cremona transformations $\PP^2_x\dashrightarrow\PP^2_y$. For seven generic point pairs, the following corollary holds.

\begin{lem}\label{cor:k=7 cremona transformations result}
Let $(x_i,y_i)_{i=1}^7$ be seven semi-generic point pairs. Then the rank of $Z = (x_i^\top \otimes y_i^\top)_{i=1}^7$ drops if and only if there are cubic curves $C_1,C_2$, through $x_1,\ldots,x_7$ and $y_1\ldots,y_7$ respectively, as well as a two-parameter family of Cremona transformations $f_{\ell}:\PP^2_x\dashrightarrow\PP^2_y$ such that $f(x_i)= y_i$ for all $i$ and the family is well-defined on the restriction $C_1\to C_2$. Furthermore, if this holds then $C_1=C_x$ and $C_2=C_y$.
\end{lem}

\begin{proof}
($\Rightarrow$) Under semi-genericity, $Z$ is rank deficient if and only if the nullspace of $Z$ and the nullspaces of each of its $6\times 9$ submatrices are identical. In particular, if $P_i$ is the subset of $6$ point pairs obtained by excluding the $i$th, then, using the notation from Theorem~\ref{thm:trinity correspondence (vision)},
$\mathcal{L}_{P_1}=\ldots=\mathcal{L}_{P_7}$. Applying Lemma~\ref{lem:k=7 cubic curves characteristics}, we find that the pairs of curves $C_x^{\hat i},C_y^{\hat i}$ are identical for all $i$. Accordingly, we exclude the indexing and identify them as $C_x$ and $C_y$ respectively. Similarly, the family of Cremona transformations $\mathcal{C}_{P_1}=\ldots=\mathcal{C}_{P_7}$, and, as in Lemma~\ref{lem:k=7 cubic curves characteristics}, restricting this family to the map $C_x\to C_y$ yields a well-defined isomorphism with the property $x_i\mapsto y_i$ for all $i$.

($\Leftarrow$) For this direction, we use Theorem~\ref{thm:trinity correspondence (vision)}. In particular, the existence of such a family of Cremona transformations implies $\text{dim}(\mathcal{L}_P)=\text{dim}(\mathcal{C}_P)=2$ as illustrated in \eqref{eq:trinity (vision)}. 
Since there is a two-dimensional family of lines $\ell$ in the projective nullspace of $Z$, 
we must have $\text{rank}(Z)<7$. We now need to verify that $C_1=C_x$ and $C_2=C_y$. It follows by Lemma~\ref{lem:base points lie on cubic} that the curves $C_1,C_2$ contain all possible base points of the Cremona transformations $f_{\ell}$. Furthermore,  by Lemma~\ref{lem:lines have well-defined cremona transformations}  the sets of all such base points in the domain and codomain is exactly the set of all possible right and left epipoles. It follows that $C_x\subset C_1$ and $C_y\subset C_2$ and therefore the curves are equal.
\end{proof}

\begin{proof}[Proof of Theorem~\ref{thm:k=7 simple}]
($\Rightarrow$) This direction follows from Lemma~\ref{cor:k=7 cremona transformations result}. In particular, the isomorphism is exactly that obtained by restricting the family of Cremona transformations to the map $C_x\to C_y$.

($\Leftarrow$) Assume such curves $C_1,C_2$ exist, as well as the desired isomorphism $C_ 1\to C_2$. By Lemma~\ref{lem:isomorphisms are cremona} there is a two-parameter family of Cremona transformations $\PP^2_x \dashrightarrow\PP^2_y$ whose restriction $C_1\to C_2$ yields this isomorphism. It follows from Lemma~\ref{cor:k=7 cremona transformations result} that $Z$ is rank deficient and that $C_1=C_x$ and $C_2=C_y$.
\end{proof}

\subsection{The Cremona Hexahedral form of $C_x$ and $C_y$} \label{subsec:Cremona hexahedral}
In this section we return to the 
original characterization of the cubic curves $C_x$ and $C_y$ as the images under the quadratic maps $\kappa_x$ and $\kappa_y$ of the 
curve $C$ as in \eqref{eq:adjoint maps}. We will see 
that it is possible to derive explicit equations for these curves using the classical theory of cubic surfaces and a special invariant-theoretic representation of them called the 
Cremona hexahedral form. These ideas intersect substantially with the characterization of rank drop of $Z_6$ 
in \cite{rankdrop1}; in particular, we draw on the connection between six generic points pairs $(x_i,y_i)_{i=1}^6$ and cubic surfaces. We will begin by explicitly characterizing the curve $C=\mathcal{N}_Z\cap\mathcal{D}$ as the planar section of a cubic surface; we will then use this characterization in conjunction with material from \cite{rankdrop1} to find explicit equations for the curves $C_x$ and $C_y$.

\subsubsection{Six generic point pairs again} 
Suppose we have six generic point pairs $(x_i,y_i)_{i=1}^6$; in particular, $Z = (x_i^\top \otimes y_i^\top)_{i=1}^6$ has full rank.
Let $Z_{\hat j}$ denote the $5\times 9$ matrix obtained by deleting the $j$th row of $Z$. Then $\mathcal{N}_{Z_{\hat{j}}} \cong \PP^3$
and $S_{\hat{j}} := \mathcal{N}_{Z_{\hat{j}}} \cap \mathcal{D}$ is a smooth cubic surface in $\mathcal{N}_{Z_{\hat{j}}}$ by the genericity assumption, and hence all points on it have rank-two. It was shown in \cite{rankdrop1} that 
$S_{\hat{j}}$ is the blowup of $\PP^2_x$ at $(\{x_i\}_{i=1}^6 \setminus \{x_j\}) \cup \{\bar{x}_j\}$ where $\bar{x}_j$ is a new point that arises from $\{x_i\}_{i=1}^6 \setminus \{x_j\}$, see Lemma 6.1 of \cite{rankdrop1} for its derivation and formula. Symmetrically, $S_{\hat{j}}$ is also the 
blowup of $(\{y_i\}_{i=1}^6 \setminus \{y_j\}) \cup \{\bar{y_j}\}$ in $\PP^2_y$ where $\bar{y_j}$ is a new point determined by $\{y_i\}_{i=1}^6 \setminus \{y_j\}$. The quadratic maps 
$\kappa_x^{\hat j} \,:\, S_{\hat{j}} \to \PP^2_x$ and $\kappa_y^{\hat j} \,:\, S_{\hat{j}} \to \PP^2_y$ are $1:1$ except on the exceptional lines of the blowup. The curve $C$ is given by 
$$C = \mathcal{N}_{Z} \cap \mathcal{D} = \mathcal{N}_{Z_{\hat{j}}} \cap \mathcal{D} \cap (x_j^\top \otimes y_j^\top)^\perp = S_{\hat{j}} \cap (x_j^\top \otimes y_j^\top)^\perp.$$
Therefore, $C$ cuts each of the exceptional lines of the blowup in one point and therefore the restrictions of $\kappa_x,\kappa_y$ to $C$ are isomorphisms.

For a set of six points $u_1,\ldots,u_6 \in \PP^2$, setting $[ijk]:= \det[u_i \, u_j \, u_k]$, define 
\begin{equation}[(ij)(kl)(rs)]:=[ijr][kls]-[ijs][klr].\end{equation}
This is a classical invariant of $u_1, \ldots, u_6$ under the action of 
$\textup{PGL}(3)$ 
whose vanishing expresses that the  lines 
$\overline{u_iu_j}$, $\overline{u_ku_l}$ and $\overline{u_ru_s}$ 
meet in a point \cite{coble}*{pp 169}. Using these invariants, Coble 
defines the following six scalars \cite{coble}*{pp 170}:
\begin{small}
\begin{align}\label{eq:bumped up Joubert invariants}
\begin{split}
\bar{a}&=[(25)(13)(46)]+[(51)(42)(36)]+[(14)(35)(26)]+[(43)(21)(56)]+[(32)(54)(16)]\\
\bar{b}&=[(53)(12)(46)]+[(14)(23)(56)]+[(25)(34)(16)]+[(31)(45)(26)]+[(42)(51)(36)]\\
     \bar{c}&=[(53)(41)(26)]+[(34)(25)(16)]+[(42)(13)(56)]+[(21)(54)(36)]+[(15)(32)(46)]\\
\bar{d}&=[(45)(31)(26)]+[(53)(24)(16)]+[(41)(25)(36)]+[(32)(15)(46)]+[(21)(43)(56)]\\
\bar{e}&=[(31)(24)(56)]+[(12)(53)(46)]+[(25)(41)(36)]+[(54)(32)(16)]+[(43)(15)(26)]\\
\bar{f}&=[(42)(35)(16)]+[(23)(14)(56)]+[(31)(52)(46)]+[(15)(43)(26)]+[(54)(21)(36)]
\end{split}
\end{align}
\end{small}

Coble also defines the  
following six cubic polynomials that vanish on $u_1, \ldots, u_6$:
\begin{small}
\begin{align} \label{eq:covariant cubics}
\begin{split}
a(u)&=[25u][13u][46u]+[51u][42u][36u]+[14u][35u][26u]+[43u][21u][56u]+[32u][54u][16u]\\
b(u)&=[53u][12u][46u]+[14u][23u][56u]+[25u][34u][16u]+[31u][45u][26u]+[42u][51u][36u]\\
c(u)&=[53u][41u][26u]+[34u][25u][16u]+[42u][13u][56u]+[21u][54u][36u]+[15u][32u][46u]\\
d(u)&=[45u][31u][26u]+[53u][24u][16u]+[41u][25u][36u]+[32u][15u][46u]+[21u][43u][56u]\\
e(u)&=[31u][24u][56u]+[12u][53u][46u]+[25u][41u][36u]+[54u][32u][16u]+[43u][15u][26u]\\
f(u)&=[42u][35u][16u]+[23u][14u][56u]+[31u][52u][46u]+[15u][43u][26u]+[54u][21u][36u]
\end{split}
\end{align}
\end{small}
These cubic polynomials are {\em covariants} of $u_1, 
\ldots, u_6$ under the action of $\textup{PGL}(3)$.

It is a well-known result in algebraic geometry that every smooth cubic surface is the blowup of six points in $\PP^2$. 
The blowup procedure furnishes an algorithm to find a 
determinantal representation of the surface. However, these representations do not directly reflect the six points that were blown up. The {\em Cremona hexahedral form} of a smooth cubic surface provides 
explicit equations for the surface in terms of the points being blown up. It consists 
of the following polynomials:
\begin{align}\label{eq:canonical surface}
    \begin{split}
    z_1^3+z_2^3+z_3^3+z_4^3+z_5^3+z_6^3&=0\\
    z_1 + z_2 + z_3 + z_4 + z_5 + z_6 & = 0\\
    \bar{a} z_1+\bar{b} z_2 +\bar{c} z_3+\bar{d} z_4 +\bar{e} z_5 +\bar{f} z_6 &=0\\
\end{split}
\end{align}
Furthermore, the cubic surface can also be parameterized by   
\begin{align} \label{eq:parametrized surface}
      \overline{\{(a(u):b(u):c(u):d(u):e(u):f(u))~:~u\in\PP^2\}}.
\end{align}

We will now use the above facts  
to obtain explicit equations (that depend on $(x_i,y_i)_{i=1}^6$), of the epipolar curves $C_x$ and $C_y$. In what follows, we subscript $\bar{a}, \ldots, \bar{f}$ and $a(u), \ldots, f(u)$ with $x$ (respectively,  $y$) when $u_i = x_i$ (respectively, $u_i = y_i$).

\begin{defn}\label{def: k=7 canoncial forms c_x and c_y}
Given six point pairs
$(x_i,y_i)_{i=1}^6$, define the following cubic polynomials:
\begin{equation}\label{eq:canonical C_x and C_y}
    \begin{split}
        g_x(u):=\bar{a}_ya_x(u)+\bar{b}_yb_x(u)+\bar{c}_yc_x(u)+\bar{d}_yd_x(u)+\bar{e}_ye_x(u)+\bar{f}_yf_x(u),\\
        g_y(v):=\bar{a}_xa_y(v)+\bar{b}_xb_y(v)+\bar{c}_xc_y(v)+\bar{d}_xd_y(v)+\bar{e}_xe_y(v)+\bar{f}_xf_y(v).
    \end{split}
\end{equation}
Given seven point pairs $(x_i,y_i)_{i=1}^7$, let $g_x^{\hat{i}}$ and $g_y^{\hat{i}}$ denote the above cubic polynomials obtained from the point pairs $(x_j,y_j)_{j\neq i}$.
\end{defn}

The polynomials $g_x,g_y$ played a prominent role in the rank drop of $Z_6$ 
in \cite{rankdrop1}.

\begin{lem}\label{lem: k=7 canonical forms for C}
Given generic point pairs $(x_i,y_i)_{i=1}^6$, let $C = \mathcal{N}_Z \cap \mathcal{D}$, $C_x = \kappa_x(C) \subset \PP^2_x$ and $C_y = \kappa_y(C) \subset \PP^2_y$. Also let $S_x$ be the blowup of $\PP^2_x$ at $x_1,\ldots,x_6$ and let $S_y$ be the 
blowup of $\PP^2_y$ at $y_1,\ldots,y_6$, each expressed in Cremona hexahedral form. Then the following hold true:
\begin{enumerate}
    \item The plane cubic curves $C_x$ and $C_y$ have defining equations $g_x(u)=0$ and $g_y(v)=0$ respectively.
    \item The cubic curve $C\cong S_x\cap S_y$ which has equations:
    \begin{align}\label{eq:canonical C}
    \begin{split}
    z_1^3+z_2^3+z_3^3+z_4^3+z_5^3+z_6^3&=0\\
    z_1 + z_2 + z_3 + z_4 + z_5 + z_6 & = 0\\
    \bar{a}_x z_1+\bar{b}_x z_2 +\bar{c}_x z_3+\bar{d}_x z_4 +\bar{e}_x z_5 +\bar{f}_x z_6 &=0\\
    \bar{a}_y z_1+\bar{b}_y z_2 +\bar{c}_y z_3+\bar{d}_y z_4 +\bar{e}_y z_5 +\bar{f}_y z_6 &=0
    \end{split}
    \end{align}
    \item The cubic curve $S_x \cap S_y$ is the image of 
    $C_x$ 
    under the blowup of $\PP^2_x$ at $x_1, \ldots, x_6$ and also the image of 
    $C_y$ under the blowup of $\PP^2_y$ at $y_1,\ldots,y_6$.
    \end{enumerate}
\end{lem}

\begin{proof}
We begin with the first item. By Lemma~\ref{lem:k=7 cubic curves characteristics}, $x_i\in C_x$ for all $i$ and by Definition~\ref{def: k=7 canoncial forms c_x and c_y}, $g_x(x_i)=0$ for all $i$ since the cubic polynomials in \eqref{eq:covariant cubics} vanish on the $x_i$. For fixed $i=1,\ldots,6$, consider the $5$ point pairs left after excluding $(x_i,y_i)$ and let $(u_i,v_i)$ be the unique new point pair (cf. Lemma 6.1 in \cite{rankdrop1}) such that the configuration 
\begin{equation}
\left\{ (x_1,y_1),\ldots,(x_6,y_6),(u_i,v_i) \right\} \setminus \{ (x_i,y_i) \}
\end{equation}
is rank deficient. For convenience, we assume without loss of generality that $i=6$. In other words, if $Z_{\hat 6}=(x_i\otimes y_i)_{i=1}^5$ then $(u_6,v_6)$ is the unique point pair such that $S_{\hat 6}=\mathcal{N}_{Z_{\hat 6}}\cap\mathcal{D}$ can be obtained both by blowing up $\PP^2_x$ in the points $x_1,\ldots,x_5,u_6$ and by blowing up $\PP^2_y$ in the points $y_1,\ldots,y_5,v_6$. It follows that the curve $C\subset S_{\hat 6}$ cuts the exceptional lines corresponding to $u_6,v_6$ exactly once each and therefore $u_6\in C_x$ and $v_6\in C_y$; it follows symmetrically that $u_i\in C_x$ and $v_i\in C_y$ for all $i=1,\ldots,6$. One can check using a computer algebra package that $g_x(u_6)=0$ and $g_y(v_6)=0$ after fixing points as in Lemma 6.1 in \cite{rankdrop1}; it follows symmetrically that $g_x(u_i)=0$ and $g_y(v_i)=0$ for all $i$. Finally, since $C_x$ and the curve cut out by $g_x$ share $12$ distinct points, they must be the same cubic curve; similarly we can conclude that $C_y$ is cut out by $g_y$. This finishes the proof of the first claim.

To prove the second and third claims, recall that $\kappa_x:C\to C_x$ is an isomorphism. Let $\kappa_x':S_x\to\PP^2_x$ and $\kappa_y':S_y\to\PP^2_y$ be the blow down morphisms. 
The Cremona hexahedral forms of $S_x$ and $S_y$ give
\begin{equation}
    S_x\cap S_y=\{z \in S_x~:~\bar a_yz_1+\ldots+\bar f_yz_6=0\}.
\end{equation}
By \eqref{eq:parametrized surface}, 
\begin{equation}
        S_x=\overline{\{(a_x(u):b_x(u):c_x(u):d_x(u):e_x(u):f_x(u))~:~u\in\PP^2\}}
\end{equation}
and since $C_x$ is cut out by $g_x(u)=0$, we get that 
\begin{equation}
\begin{split}
        S_x\cap S_y & =\overline{\{(a_x(u):\ldots:f_x(u))~:~
        \bar a_y a_x(u) +\ldots+\bar f_y f_x(u) = 0, u\in\PP^2_x\}}\\
                 & = \overline{\{(a_x(u):\ldots:f_x(u))~:~ u \in C_x\}}.
\end{split}
\end{equation}
Therefore, $S_x\cap S_y$ is exactly the image of $C_x$ under the blowup of $\PP^2_x$ 
at $x_1\ldots,x_6$. Restricting $\kappa_x$ to $\kappa_x'|_{S_x\cap S_y}:S_x\cap S_y\to C_x$ we obtain an isomorphism, and we have $S_x\cap S_y\cong C_x\cong C$, which proves the second claim. Finally, we note that by a symmetric argument, $S_x\cap S_y$ is also exactly the image of 
$C_y$ under the blowup of $\PP^2_y$ at $y_1,\ldots,y_6$ proving the 
third claim as well. 
\end{proof}

\begin{ex}[Example~\ref{ex: k=7 example 1}, continued] 
One can verify that the polynomials \eqref{eq:canonical C_x and C_y} define the same cubic curves as those in Example~\ref{ex: k=7 example 1}. We then pick a specific point $x_7=(0:1403:118)\in C_x$. Using a computer algebra package, one can compute the unique point 
$y_7=(1802855:1562942:171287)$ such that the $Z = (x_i,y_i)_{i=1}^7$ is rank deficient. It is straight-forward to 
verify that $y_7 \in C_y$.  Moreover, there is a two-parameter family of Cremona transformations $f_\ell$ such that $x_i\mapsto y_i$ for $i=1,\ldots, 6$ and for all members of this family $f_\ell(x_7)=(1802855:1562942:171287)$, which lines up with Lemma~\ref{cor:k=7 cremona transformations result}. These points can be seen on the cubic curve in Figure~\ref{fig: k=7 cubic curves 2}.

\begin{figure}
\fbox{%
\includegraphics[scale=0.1353]{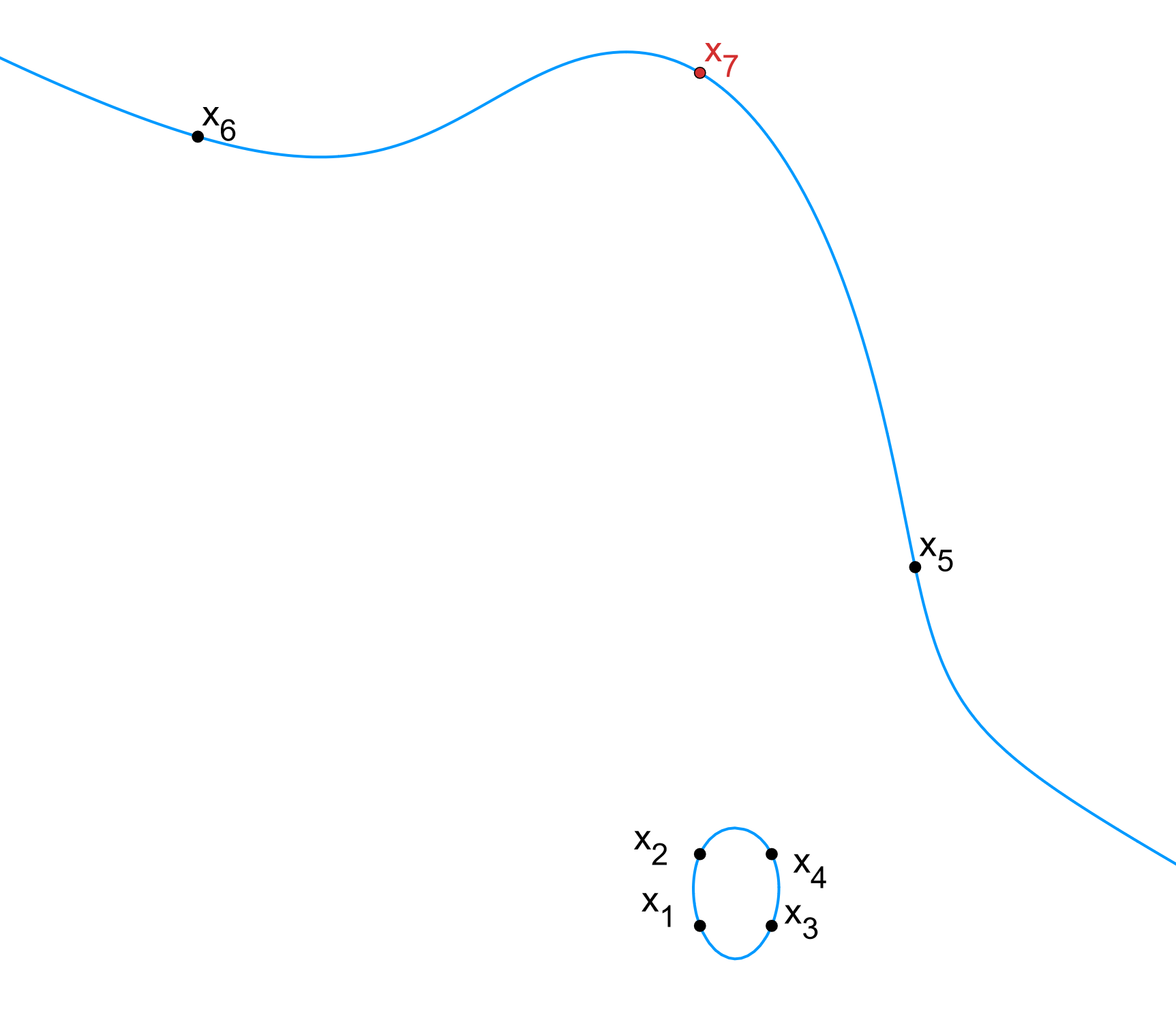}
}
\fbox{%
\includegraphics[scale=0.137]{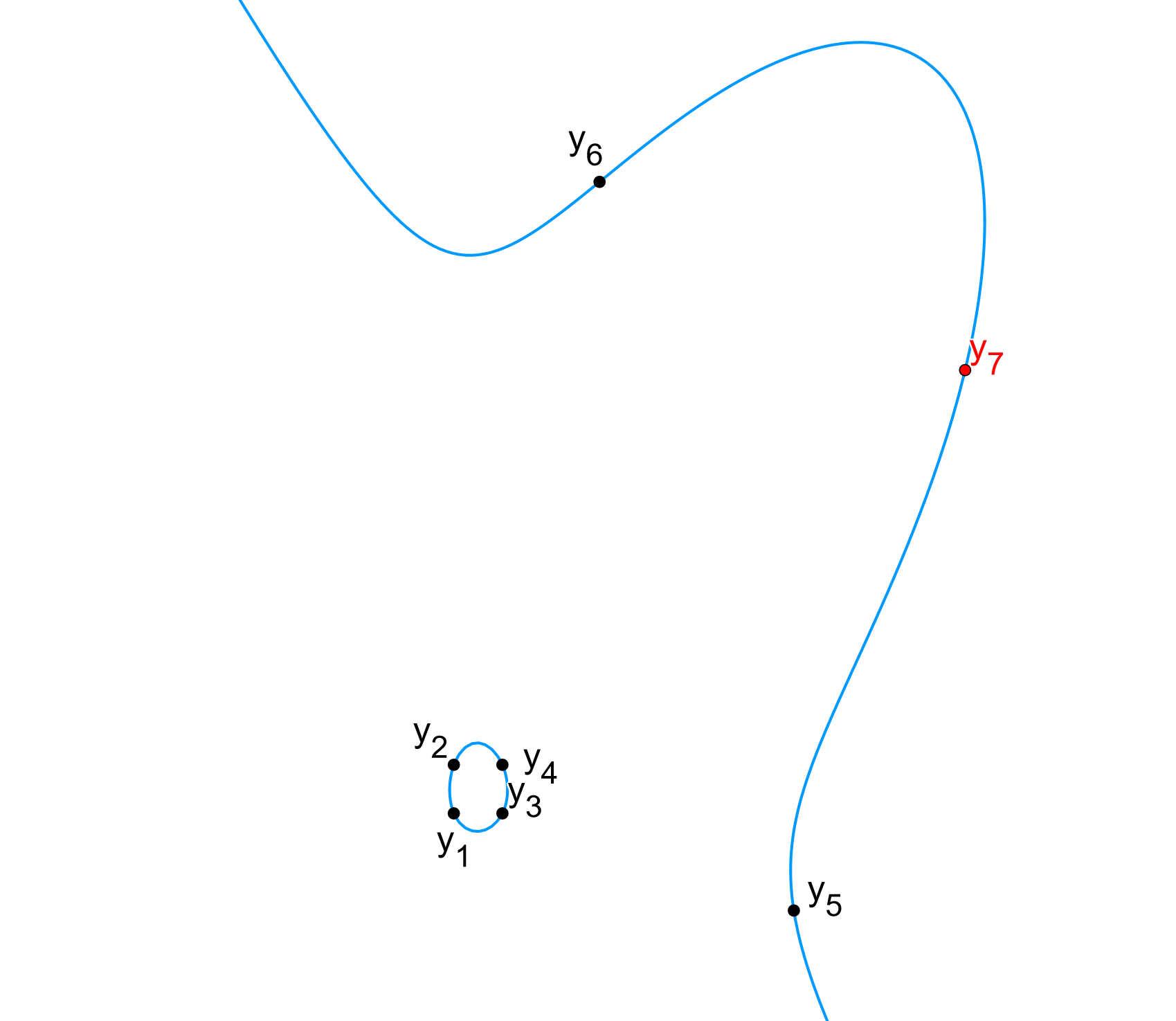}
}
\caption{The cubic curves $C_x$ and $C_y$, with $x_7$ and $y_7$ highlighted.}\label{fig: k=7 cubic curves 2}
\end{figure}
\end{ex}

\subsubsection{Algebraic Conditions for the rank deficiency of $Z_7$} \label{subsec:k=7 full proof} 
We are now ready to present our main algebraic result for rank drop given $k=7$ point pairs. We begin with a basic lemma that will connect all of our results in the main theorem.

\begin{lem}\label{lem: k=7 basic rank drop description}
Let $(x_i,y_i)_{i=1}^7$ be seven semi-generic points. 
Then $Z=(x_i^\top\otimes y_i^\top)_{i=1}^7$ is rank deficient if and only if $C^{\hat{1}}=\cdots=C^{\hat{7}}$
where $C^{\hat{i}}$ is the cubic curve $\mathcal{N}_{Z_{\hat{i}}} \cap \mathcal{D}$.
\end{lem}
\begin{proof}
By semi-genericity, $Z$ is rank deficient if and only if 
$\mathcal{N}_Z = \mathcal{N}_{Z_{\hat{1}}}=\cdots=\mathcal{N}_{Z_{\hat{7}}}$
for each 
$6\times9$ submatrix  $Z_{\hat{i}}$ of $Z$. Since   $C^{\hat{i}}=\mathcal{N}_{Z_{\hat{i}}}\cap \mathcal{D}$, 
$Z$ is rank deficient if and only if $C = C^{\hat{1}}=\cdots=C^{\hat{7}}$.
\end{proof}

The following theorem, which is the main result of this subsection, allows us to check for rank drop without computing Cremona transformations.

\begin{thm}\label{thm: k=7 central theorem}
For $(x_i,y_i)_{i=1}^7$ semi-generic, the following are equivalent:
\begin{enumerate}
    \item $Z = (x_i^\top \otimes y_i^\top)_{i=1}^7$ is rank deficient.
    \item We have $x_i\in C_x^{\hat {i}}$ and $y_i\in C_y^{\hat{i}}$ for all $i=1,\ldots, 7$.
    \item We have $g_x^{\hat{i}}(x_i)=0$ and $g_y^{\hat{i}}(y_i)=0$ for all $i=1,\ldots,7$.
    \item All $7$ cubic curves in $\PP^2_x$ are equal: $C_x^{\hat{7}}=\ldots=C_x^{\hat{1}}$. 
    \item All $7$ cubic curves in $\PP^2_y$ are equal: $C_y^{\hat{7}}=\ldots=C_y^{\hat{1}}$. 
\end{enumerate}
\end{thm}
\begin{proof} By Lemma~\ref{lem: k=7 canonical forms for C}, 
$(2)\iff(3)$. We next prove that $(1)\implies(4),(5)$. If $Z$ is rank deficient, then by Lemma~\ref{lem: k=7 basic rank drop description}, $C^{\hat{1}}=\cdots=C^{\hat{7}}$. Applying the quadratic maps $\kappa_x$ and $\kappa_y$, we obtain $(4)$ and $(5)$. To prove the reverse direction we will show $(4)\implies(1)$; the proof that $(5)\implies(1)$ is symmetric. In particular, we will show that $C_x^{\hat i}=C_x^{\hat j}$ if and only if $C^{\hat i}=C^{\hat j}$. For ease of notation, we will assume $i=6$ and $j=7$. Consider the five point pairs $(x_i,y_i)_{i=1}^5$ and the matrix $Z_5=(x_i^\top \otimes y_i^\top)_{i=1}^5$. Then $S=\mathcal{N}_{Z_5}\cap\mathcal{D}$ is a cubic surface and $\kappa_x:S\to\PP^2_x$ and $\kappa_y:S\to\PP^2_y$ are $1:1$ except on the six exceptional lines in each case. Moreover, we can obtain the cubic curves $C^{\hat 6}$ and $C^{\hat 7}$ by intersecting this surface with a plane. We can conclude that $\kappa_x(C^{\hat 6})=\kappa_x(C^{\hat 7})$ only if $C^{\hat 6}=C^{\hat 7}$. It then follows that $(4)\implies (1)$ and, symmetrically, $(5)\implies(1)$.

We now prove $(1)\implies(2)$. Fix $i\in\{1,\ldots,7\}$. Then $x_j\in C^{\hat{i}}_x$ for all $j\neq i$ by Lemma~\ref{lem:k=7 cubic curves characteristics}. Moreover, since $C^{\hat{i}}_x=C^{\hat{j}}_x$ by hypothesis it follows that $x_i\in C^{\hat{i}}_x$. The other equalities follow symmetrically.

Finally, we prove $(2)\implies(1)$. 
Since $x_j \in C_x^{\hat i}$ and $y_j\in C_y^{\hat i}$ for $j\neq i$ by construction, the additional hypothesis (2) gives that  $x_1,\ldots,x_7\in\cap_{i=1}^7 C_x^{\hat i}$ and $y_1,\ldots,y_7\in\cap_{i=1}^7 C_y^{\hat i}$. 
We fix the first five point pairs $(x_i,y_i)_{i=1}^5$ and consider the $5\times 9$ matrix $Z_5=(x_i^\top \otimes y_i^\top)_{i=1}^5$. Consider the cubic surface $S=\mathcal{N}_{Z_5}\cap\mathcal{D}$ paired with the maps $\kappa_x$ and $\kappa_y$. The cubic curves $C^{\hat 6}$ and $C^{\hat 7}$ are obtained by intersecting $S$ with a plane. By genericity, the four matrices $\kappa_x^{-1}(x_6),\kappa_x^{-1}(x_7), \kappa_y^{-1}(y_6),\kappa_y^{-1}(y_7)$ are all distinct. Moreover, they are all contained in
\begin{equation}
C^{\hat 6}\cap C^{\hat 7}=(\mathcal{N}_{Z^{\hat 6}}\cap\mathcal{D})\cap(\mathcal{N}_{Z^{\hat 7}}\cap\mathcal{D})=\mathcal{N}_Z\cap\mathcal{D}
\end{equation}
which can also be realized as the intersection of the cubic surface $S$ with two planes. 
If $\mathcal{N}_Z$ were one-dimensional, it would intersect $\mathcal{D}$ is at most three points.  Since we have found $4>3$ distinct points in $\mathcal{N}_Z\cap\mathcal{D}$, 
 $\mathcal{N}_Z$ must have projective dimension $\geq 2$, implying $(1)$.

\end{proof}
\subsection{Beyond semi-genericity}
Given seven semi-generic point pairs $(x_i,y_i)_{i=1}^7$, we have now fully characterized the conditions under which the matrix $Z_7$ will be rank deficient. This characterization was given geometrically (Theorem \ref{thm:k=7 simple}) and then algebratized using $14$ polynomials (Theorem \ref{thm: k=7 central theorem}). We now move away from the assumptions of semi-genericity. We will first examine how $Z_7$ becomes rank deficient without these assumptions and, to some extent, generalize our algebraic condition (Theorem \ref{thm: k=7 central theorem}) to this case. We will also consider configurations where $(x_i,y_i)_{i=1}^7$ are fully generic, and therefore $Z_7$ must have full rank; in this case, we can use the cubic curves $C_x^{\hat i},C_y^{\hat{i}}$ and their associated polynomials to characterize the epipoles of the possible fundamental matrices in terms of classical invariants.

We begin by presenting two relatively simple, but highly degenerate, conditions for the rank deficiency of $Z_7$. One of these conditions is that $Z_7$ will be rank deficient if $\{x_i\}$ and $\{y_i\}$ are equal up to a change of coordinates.
\begin{lem}
Suppose we have point pairs $(x_i,y_i)_{i=1}^7$ and an invertible projective transformation $H$ such that $Hx_i=y_i$ for all $i$. Then $Z = (x_i^\top \otimes y_i^\top)_{i=1}^7$ is rank deficient.
\end{lem} \label{lem:homography causes rank drop}
\begin{proof}
Since rank drop is a projective invariant, we can assume $x_i=y_i$ for all $i$. Then the equations  $y_i^\top F x_i= x_i^\top F x_i = 0, i=1,\ldots,7$  hold for all $3 \times 3$ skew-symmetric matrices $F \in \textup{Skew}_3$. Since $\textup{Skew}_3$ is a three-dimensional vector space, $\dim(\mathcal{N}_{Z}) \geq 3$ and $\text{rank}(Z)\leq 9-3=6$.
\end{proof}

The second simple condition is that the rank of $Z$ will drop if the points in either $\PP^2$ lie in a line.

\begin{lem}\label{lem: k=7 one side in a line}
Suppose $(x_i,y_i)_{i=1}^7$ is such that either $\{x_i\}$ or $\{y_i\}$ are in a line. Then $Z = (x_i^\top \otimes y_i^\top)_{i=1}^7$ is rank deficient.
\end{lem}

\begin{proof}
Suppose the $y_i$'s are in a line. Then we may assume that $y_i=(m_i,0,1)$ after a change of coordinates. Then simple column operations on $Z$ show that it is rank deficient. 
\end{proof}
\begin{rem}
We note that the existence of such configurations does not necessarily imply that the rank drop variety is reducible. We suspect that these configurations are in the Zariski closure of the generic rank drop component.
\end{rem}
It is simple to check that in both of the above cases we will have $g_x^{\hat i}(x_i)=0=g_y^{\hat i}(y_i)$ for all $i=1,\ldots,7$, suggesting a possible generalization of Theorem \ref{thm: k=7 central theorem}(3). This is possible to some extent. In particular, even without any genericity assumptions, if $Z_7$ is rank deficient then these $14$ polynomial equations hold.

\begin{lem}\label{lem: k=7 If rank drop then 14 equations}
If $Z = (x_i^\top \otimes y_i^\top)_{i=1}^7$ is rank deficient then $g_x^{\hat{i}}(x_i)=0$ and $g_y^{\hat{i}}(y_i)=0$ for all $i$.
\end{lem}
\begin{proof}
Let $I$ be the ideal generated by the $14$ polynomials 
$g_x^{\hat{i}}(x_i)$ and $g_y^{\hat{i}}(y_i)$ for $i=1,\ldots,7$
in the polynomial ring $\CC[x_{ij},y_{ij} : i=1, \hdots, 7, j=1,2,3]$,
treating $(x_i,y_i)_{i=1}^7$ as symbolic. If $Z$ is the appropriate symbolic $7\times 9$ matrix then it can be verified using Macaulay2 that $I$ is contained in the ideal generated by the maximal minors of $Z$. 
\end{proof}

However, the converse does not hold in general. We present two examples of highly degenerate configurations where the $14$ equations hold, but $Z_7$ is not rank deficient.

\begin{ex}\label{ex:k=7 degenerate 1}
Take $x_i$ to be the columns of the matrix $X$ and $y_i$ to be the columns of the matrix $Y$
\begin{equation}
X=\begin{bmatrix} 0 & 1 & 3 & 4 & 0 & 0 & 7\\ 0 & 0 & 0 & 0 & 1 & 1 & 0\\ 1 & 0 & 1 & 1 & 0 & 0 & 1\end{bmatrix}\quad\quad Y=\begin{bmatrix} 0 & 1 & 4 & 0 & 9 & 1 & 0\\ 0 & 0 & 0 & 1 & 0 & 0 & 1\\ 1 & 0 & 1 & 0 & 1 & 1 & 0\end{bmatrix}
\end{equation}
where $x_1,x_2,x_3,x_4,x_7$ are on a line and $x_5=x_6$. Similarly, $y_1,y_2,y_3,y_5,y_6$ are on a line and $y_4=y_7$. We can verify that $g_x^{\hat{i}}(x_i)=0=g_y^{\hat {i}}(y_i)$ for all $i=1,\ldots,7$ and that the matrix $Z$ is not rank deficient. In particular, $\mathcal{N}_Z$ is spanned by the singular matrices
$$\begin{bmatrix}0 & 0 & -3\\ 0 & 0 & 0\\ 4 & 0 & 0\end{bmatrix}\quad\quad\begin{bmatrix}0 & 0 & 0\\ 0 & 1 & 0\\ 0 & 0 & 0\end{bmatrix}$$
the latter of which has rank one.
\end{ex}
\begin{ex}\label{ex:k=7 degenerate 2}
Take $x_i$ to be the columns of the matrix $X$ and $y_i$ to be the columns of the matrix $Y$
\begin{equation}
X=\begin{bmatrix} 1 & 2 & 5 & 1 & 2 & 3 & 7\\ 0 & 0 & 0 & 0 & 1 & 2 & 6\\ 1 & 1 & 1 & 0 & 1 & 1 & 1\end{bmatrix}\quad\quad Y=\begin{bmatrix} 0 & 0 & 0 & 0 & 1 & 3 & 4\\ 1 & 5 & 1 & 0 & 2 & 6 & 8\\ 1 & 1 & 0 & 1 & 1 & 1 & 1\end{bmatrix}
\end{equation}
where $\{x_i\}_{i=1}^4,\{y_i\}_{i=1}^4$ and $\{x_i\}_{i=5}^7,\{y_i\}_{i=5}^7$ are on distinct lines in each image. We can verify that $g_x^{\hat{i}}(x_i)=0=g_y^{\hat {i}}(y_i)$ for all $i=1,\ldots,7$ and that the matrix $Z$ is not rank deficient. In particular, $\mathcal{N}_Z$ is spanned by the rank one matrices
$$\begin{bmatrix}0 & -2 & 0\\ 0 & 1 & 0\\ 0 & 0 & 0\end{bmatrix}\quad\quad\begin{bmatrix}-1 & 1 & 1\\ 0 & 0 & 0\\ 0 & 0 & 0\end{bmatrix}.$$
\qed
\end{ex}

While the focus of this paper has been on the conditions under which $Z$ drops 
rank, the tools we have developed have applications beyond rank drop. In particular, for a fully generic configuration of seven point pairs we can use the cubic curves $C_x^{\hat i}$ and $C_y^{\hat i}$ to find the possible epipoles of fundamental matrices. While this has minimal practical application, it is significant in that the characterization is entirely in terms of classical projective invariants.

\begin{lem}\label{lem:epipoles are intersection}
Let $(x_i,y_i)_{i=1}^7$ be generic point pairs. In particular, we assume $\mathcal{N}_{Z}$ is one dimensional and contains three rank-two matrices $F_1,F_2,F_3$, two of which may be complex. Then the epipoles of these fundamental matrices, $e_1^x,e_2^x,e_3^x$ and $e_1^y,e_2^y,e_3^y$, can be obtained as the unique three points in the intersections $\cap_{i=1}^7 C_x^{\hat i}\subset\PP^2_x$ and $\cap_{i=1}^7 C_y^{\hat i}\subset\PP^2_y$.
\end{lem}
\begin{proof}
Consider the two cubic curves $C_x^{\hat 7}, C_x^{\hat 6}$. The intersection $A_{6,7}=C_x^{\hat 7}\cap C_x^{\hat 6}$ will contain exactly nine points. We know that $x_1,\ldots,x_5\in A_{6,7}$. Additionally, let $(u_6,v_6)$ be the pair of rank drop points, as in Lemma 5.1 of \cite{rankdrop1}, associated to  $(x_i,y_i)_{i=1}^5$. Then, by Lemma~\ref{lem: k=7 If rank drop then 14 equations}, $u_6\in A_{6,7}$ as well. There should be three more points in the intersection. Let $f$ be the unique Cremona transformation  $f:\PP^2_x\dashrightarrow\PP^2_y$ such that $x_i\mapsto y_i$ for $i=1,\ldots,7$. This $f$ is contained in the two parameter family of Cremona transformations $\PP^2_x\dashrightarrow\PP^2_y$ such that $x_i\mapsto y_i$ for all $i=1,\ldots,6$. By Lemma~\ref{lem:k=7 cubic curves characteristics} the base points of $f$ are contained in $C_x^{\hat 7}$. By a symmetric argument these base points are also contained in $C_x^{\hat 6}$ and we can conclude that these three base points are the last three points in the intersection. By Lemma~\ref{lem:lines have well-defined cremona transformations} these base points are exactly the epipoles of the fundamental matrices, and it follows by symmetry that $e_x^1,e_x^2,e_x^3\in\cap_{i=1}^7 C_x^{\hat i}$. Clearly the points $x_1,\ldots,x_5,u_6\not\in\cap_{i=1}^7 C_x^{\hat i}$ generically, and thus these three base points are the unique points in the intersection of all seven cubic curves. Symmetrically, $e_1^y,e_2^y,e_3^y$ are the unique points in $\cap_{i=1}^7 C_y^{\hat{i}}$.
\end{proof}
\begin{ex}
Take $x_i$ to be the columns of the matrix $X$ and $y_i$ to be the columns of the matrix $Y$.
\begin{equation}
X=\begin{bmatrix}
3 & 2 & 5 & 0 & 4 & -20 & -4\\ 
0 & 7 & 3 & 3 & 2 & 25 & 7\\ 
1 & 1 & 2 & 1 & 5 & 12 & 2\end{bmatrix}\quad\quad
Y=\begin{bmatrix}
0 & -49 & -15 & -3 & -5 & 5 & 7\\ 
-1 & 14 & 25 & 0 & 10 & 4 & 4\\ 
1 & 9 & 4 & 1 & 6 & 2 & 1\end{bmatrix}
\end{equation}
 We can then construct the seven cubic curves $C_x^{\hat{i}}$ and $C_y^{\hat{i}}$ in each $\PP^2$. See Figure~\ref{fig:7 cubic curves in each image}. Each set of seven cubic curves has three common intersection points. If we compute $\mathcal{N}_Z$ we find that there are exactly three possible real fundamental matrices. These matrices have epipoles
 \begin{figure}
\fbox{%
\includegraphics[scale=1.8]{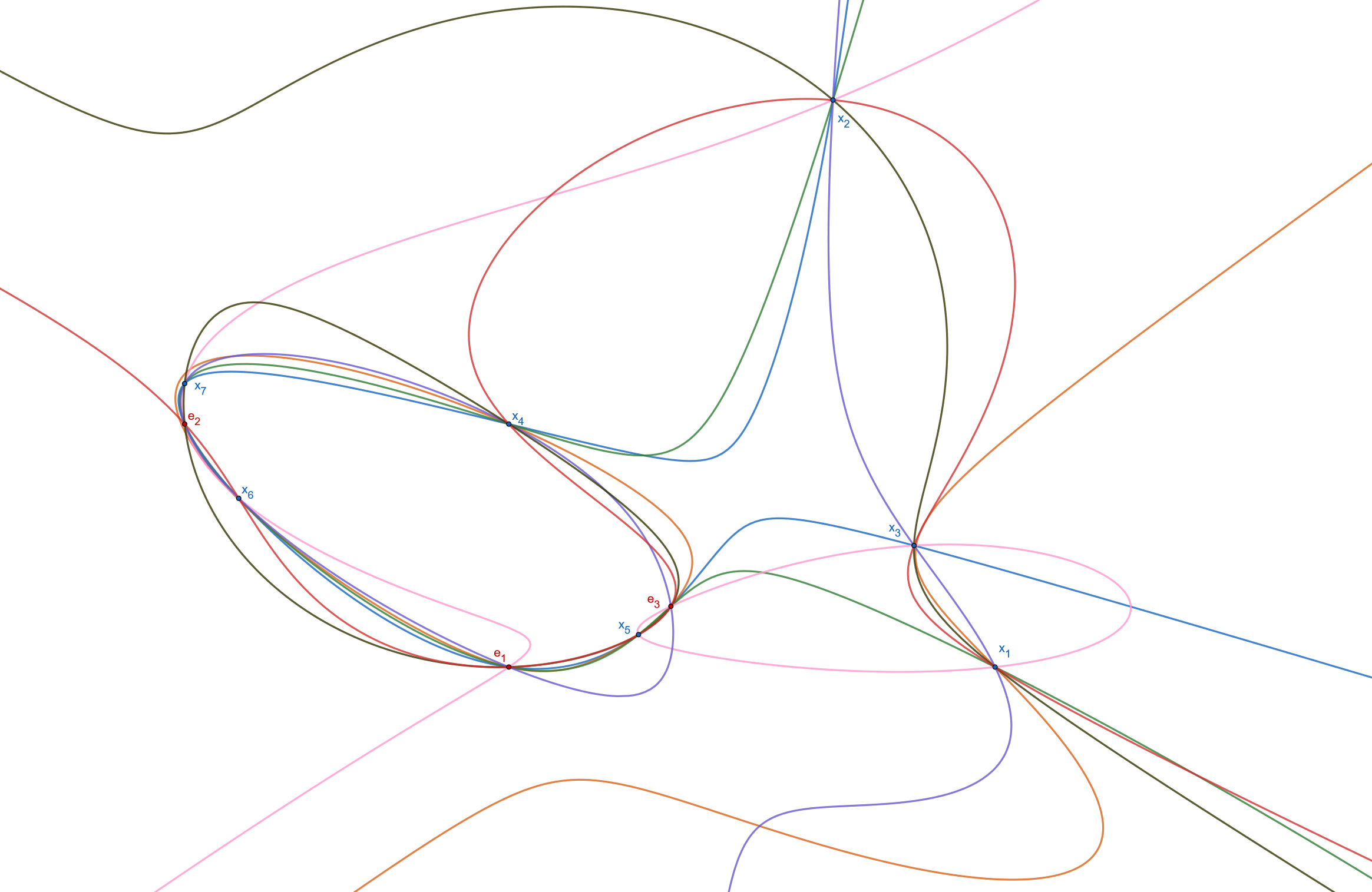}
}
\fbox{%
\includegraphics[scale=0.9]{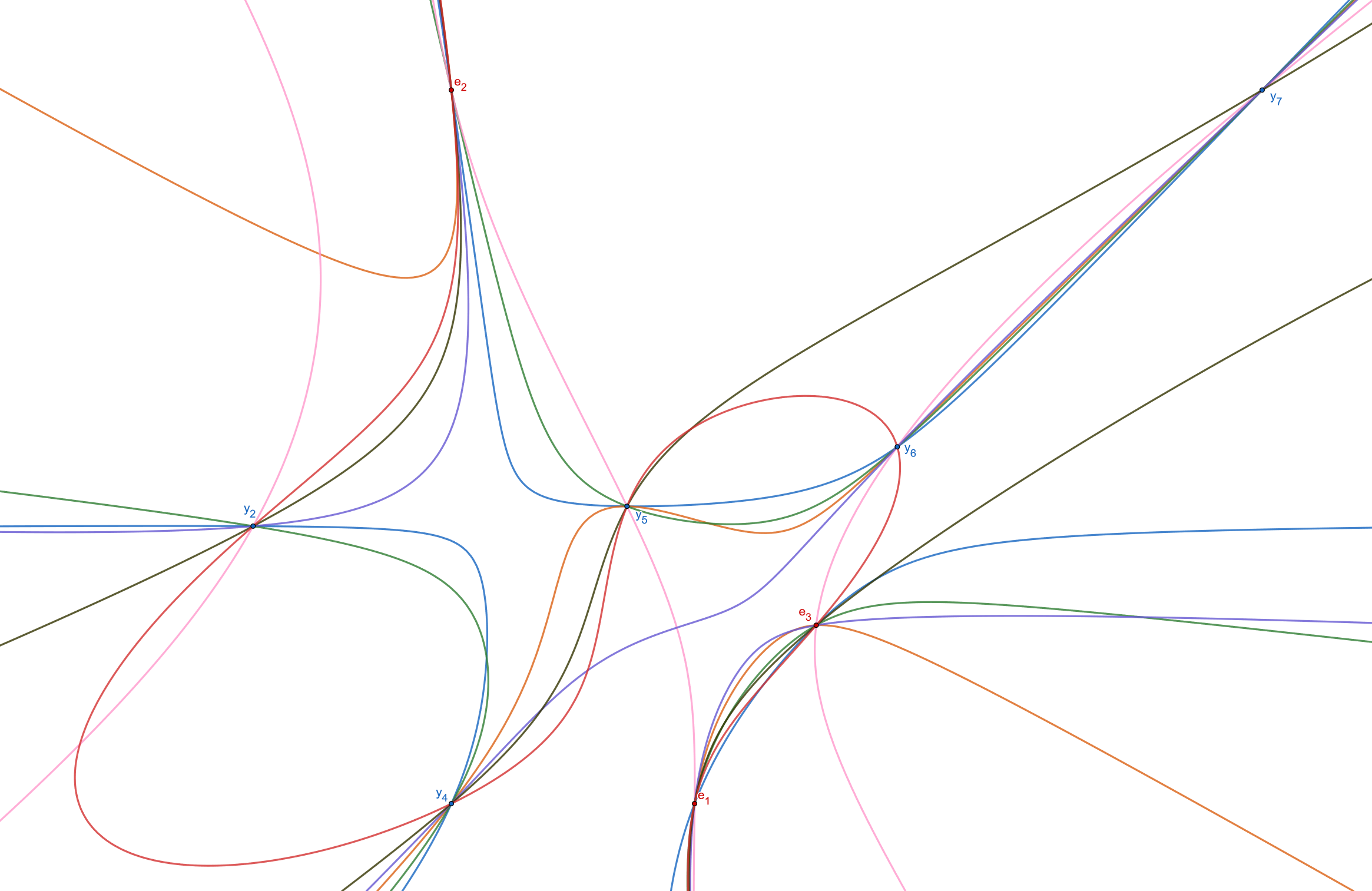}
}
\caption{The cubic curves $C_x^{\hat{i}}$ and $C_y^{\hat{i}}$. The intersection points are exactly the three possible epipoles associated to the fundamental matrices.}\label{fig:7 cubic curves in each image}
\end{figure}
\begin{equation}
\begin{split}
e_x^1=(0:0:1)\quad\quad &e_y^1=(0:0:1)\\
e_x^2=(-2:3:1)\quad\quad &e_y^2=(-3:4:1)\\
e_x^3=(4:3:4)\quad\quad &e_y^3=(3:2:2)
\end{split}
\end{equation}
and we can see that these are exactly the three common intersection points.
\end{ex}


\section{$k=9$}\label{sec:k=9}

We finish by characterizing the rank deficiency of 
$Z =  (x_i^\top \otimes y_i^\top)_{i=1}^9$, and this time we make no 
assumptions on the point pairs $(x_i,y_i)_{i=1}^9$. A simple algebraic characterization of rank drop in this case is that $\det(Z)=0$. This is a single polynomial equation but as mentioned already, typically this equation does not shed much light on the geometry of the points $\{x_i\}$ and $\{y_i\}$ that makes $Z$ rank deficient. By the methods of invariant theory, it 
is possible to write $\det(Z)$ as a polynomial in the brackets $[ijk]_x$ and $[ijk]_y$ constructed from $\{x_i\}$ and $\{y_i\}$ which may or may not offer geometric insight. Below we provide a 
geometric characterization of rank drop in terms of the two point sets in $\PP^2_x$ and $\PP^2_y$. The result is straight-forward.

Recall that if $a, b$ are distinct points in $\PP^2$, then $a \times b \in \PP^2$ is the normal of the line containing $a$ and $b$, i.e., 
$u \in \textup{Span}\{a,b\}$ if and only if 
$u^\top (a \times b) = 0$. In what follows we let $\ell_{ab}$ denote the line spanned by $a,b$. Its normal 
$a \times b = [a]_\times b$ where 
$[a]_\times$ is the $3 \times 3$ skew symmetric matrix that expresses cross products with $a$ as a matrix-vector multiplication.

\begin{thm} \label{thm:k=9}
The matrix $Z=(x_i^\top \otimes y_i^\top)_{i=1}^9$ is rank deficient if and only if  there is a projective transformation $T\,:\, \PP^2_x \dashrightarrow \PP^2_y$ such that $y_i^\top (Tx_i) = 0$ for all $i=1, \ldots, 9$, or equivalently, $y_i$ lies on the line with normal vector $Tx_i$ for all $i=1,\ldots,9$. 

This manifests  in three possible ways depending on the rank of $T$:
\begin{enumerate}
    \item There exists a line $\ell \subset \PP^2_x$ and a line $\ell' \subset \PP^2_y$ such that for each $i$, 
    either $x_i \in \ell$ or $y_i \in \ell'$ (both may happen for a given $i$). 
    \item There are two points $e \in \PP^2_x$ and $e' \in \PP^2_y$ and a $\PP^1$-homography sending the pencil of lines through $e$ to the pencil of lines through $e'$ such that 
    $\ell_{e x_i} \mapsto \ell_{e' y_i}$ for each $i$.
    \item There is some 
    $T \in \textup{PGL}(3)$ such that  
    $y_i$ lies on the line with normal vector 
    $Tx_i$ for each $i$. 
\end{enumerate}
\end{thm}

\begin{proof}
The first statement is trivial. The matrix $Z$ is rank deficient if and only if $\mathcal{N}_Z \subset \PP^8$ contains at least one point. 
Representing such a point by $T \in \PP(\CC^{3 \times 3})$ we have 
$(x_i^\top \otimes y_i^\top) \textup{vec}(T) = y_i^\top (Tx_i)=0$ for all $i=1, \ldots,9$.

\begin{enumerate}
\item 
If $\rank(T) = 1$, then $T = uv^\top$ for some $u,v \in \CC^3$. 
Therefore, $(y_i^\top u) (v^\top x_i) = 0$ for all $i=1, \ldots 9$ which is equivalent to saying that 
for each $i$, at most one of 
$u^\top y_i$ or $v^\top x_i$ can be non-zero. Therefore there exists lines $\ell$ (with normal $v$) and $\ell'$ (with normal $u$) 
such that for each $i$, either $x_i \in \ell$ or 
$y_i \in \ell'$.

\item Suppose $\rank(T)=2$. 
Let $e \in \PP^2_x$ be the unique point in the right nullspace of $T$ and $e' \in \PP^2_y$ be the unique point in the left nullspace of $T$.  The pencil of all lines through $e$ (respectively $e'$) can be identified with $\PP^1$. 

Pick any line $\ell$ not passing through $e$ and suppose its normal is $n$. Then 
the projective transformation 
$T [n]_\times$ is a $\PP^1$-homography that takes $\ell_{e x_i} \to \ell_{e' y_i}$ \cite[Result 9.5]{hartley-zisserman-2004}. 
Indeed, suppose the intersection of $\ell$ and $\ell_{e x_i}$ is 
$u_i$. Since $u_i$ is orthogonal to both $n$ and $e \times x_i$, 
we have that $u_i \sim n \times (e \times x_i) = [n]_\times (e \times x_i)$. 
Since $u_i$ lies on $\ell_{e x_i}$, we have that $u_i = \lambda e + \mu x_i$ for some scalars $\lambda, \mu$,  and since $\ell$ does not contain $e$, $u_i \neq e$ which implies that $\mu \neq 0$. 
Therefore, 
$$T [n]_\times (e \times x_i) = T u_i = \lambda T e + \mu T x_i =  0 + \mu Tx_i \sim Tx_i$$ 
which says that the normal of $\ell_{e x_i}$ is mapped to $Tx_i$ by 
$T[n]_\times$. We just need to argue that $Tx_i$ is the normal 
of $\ell_{e' y_i}$ to finish the proof.  For this check that 
$(e')^\top Tx_i = 0$ since $(e')^\top T =0$ and 
$y_i^\top T x_i = 0$ by assumption. Therefore the line spanned by 
$e'$ and $y_i$ has normal $Tx_i$.

\item If $\rank(T) = 3$ then $T$ is a homography (invertible projective transformation). 
Then $y_i^\top Tx_i = 0$ for all $i = 1, \ldots, 9$ implies that 
$y_i$ lies on the line with normal $Tx_i$ for each $i$. 
\end{enumerate}

\end{proof}

\begin{rem}
In the proof of (2), if 
$x_i = e$ for some $i$ then $[e]_\times e = 0$ and similarly, if $y_j = e'$ for some $j$ then 
$[e']_\times y_j = 0$. Therefore, the $\PP^1$-homography will not work for the indices $i,j$ where $x_i = e$ or $y_j = e'$. 
\end{rem}

\begin{rem}
    As we saw, if seven of the nine points on either side are in a line then the rank of $Z_9$ will drop. Condition (1) allows for the situations where $3 \leq s \leq 6$ points on one side are in a line and the $9-s$ complementary $y$ points are in a line.
\end{rem}

\begin{ex}

\begin{enumerate}
    \item Take $x_i$ to be the columns of the matrix $X$ and $y_i$ to be the columns of the matrix $Y$:
    \begin{align}
        X = \begin{bmatrix}
     0&0&0&0&1&-1&1&1&1\\
     1&1&1&-1&1&1&0&1&-1\\
     0&1&2&1&0&1&1&1&-1
     \end{bmatrix}
\quad \quad 
        Y = \begin{bmatrix}
     -1&1&0&0&1&1&1&-1&1\\
     0&1&-1&1&0&0&0&0&0\\
     2&1&1&1&0&1&2&1&3
     \end{bmatrix}.
    \end{align}
    One can check that all $8 \times 9$ submatrices of $Z$ have rank $8$. If the coordinates of $\PP^2$ are $u_1,u_2,u_3$ then $x_1, \ldots, x_4$ lie on the line $u_1 = 0$ and $y_5, \ldots, y_9$ lie on the line $u_2=0$ and $Z$ must drop rank by Condition (1). Indeed, the unique element in the nullspace of $Z$ is the rank-one 
    matrix
    \begin{align}
    T = \begin{bmatrix} 0 & 0 & 0 \\ 1 & 0 & 0 \\ 0 & 0 & 0 \end{bmatrix}.
    \end{align}
    
    \item Take $x_i$ to be the columns of the matrix $X$ and $y_i$ to be the columns of the matrix $Y$:
    \begin{align}
        X = \begin{bmatrix}
      1&0&0&1&1&1&0&1&2\\
      0&1&0&1&1&0&1&2&1\\
      0&0&1&1&0&1&1&1&1
      \end{bmatrix}
      \quad\quad
      Y = \begin{bmatrix}
      1&0&0&1&1&0&1&2&1\\
      0&1&0&1&0&1&1&1&4\\
      0&0&1&1&1&1&0&1&3
      \end{bmatrix}.
\end{align}
Again, $Z$ and all its $8 \times 9$ submatrices  have rank $8$. The unique element in $\mathcal{N}_Z$ is the rank-two matrix
\begin{align}
    T =\begin{bmatrix}
      0&0&-1\\
      0&0&1\\
      -1&1&0
      \end{bmatrix}.
\end{align}
 The points $e = e' = (1,1,0)^\top$ are generators of the right and left nullspaces of $T$. 
 Note that $x_5 = e$ and $y_7 = e'$. 
 Pick $\bar{\ell} = (1,2,3)^\top$. 
 Then $e^\top \bar{\ell} \neq 0$. 
 Now check that $[e']_\times Y = (T [\bar{\ell}]_\times) [e]_\times X$. Indeed, 
 $$ [e]_\times = [e']_\times = \begin{bmatrix}
      0&0&1\\
      0&0&-1\\
      -1&1&0\end{bmatrix}, \quad \quad [\bar{\ell}]_\times = \begin{bmatrix}
      0&-3&2\\
      3&0&-1\\
      -2&1&0\end{bmatrix}, \textup{ and } $$
      \begin{align} \begin{split}
      [e']_\times Y
      &=\begin{bmatrix}
      0&0&1&1&1&1&0&1&3\\
      0&0&-1&-1&-1&-1&0&-1&-3\\
      -1&1&0&0&-1&1&0&-1&3\end{bmatrix}\\
      &\sim \begin{bmatrix}
      0&0&3&3&0&3&3&3&3\\
      0&0&-3&-3&0&-3&-3&-3&-3\\
      3&-3&0&0&0&3&-3&-3&3\end{bmatrix}
      = (T [\bar{\ell}]_\times) [e]_\times X
      \end{split}
      \end{align}
      except in the columns of $X$ and $Y$ where $x_i = e$ and $y_j = e'$.

      Here is another example where the epipoles do not appear among the $x_i$'s or $y_j$'s.
      Take $x_i$ to be the columns of the matrix $X$ and $y_i$ to be the columns of the matrix $Y$:
\begin{align}
X = \begin{bmatrix}
      1&1&1&1&1&1&1&1&1\\
      0&1&0&1&2&0&2&-1&-1\\
      0&0&1&1&0&2&1&1&-1\end{bmatrix}, \quad \quad 
Y = \begin{bmatrix}
      1&1&1&1&1&1&1&1&1\\
      0&1&0&1&0&2&1&1&-2\\
      0&0&1&1&2&1&2&-1&-1\end{bmatrix}.
\end{align}
 The unique element in $\mathcal{N}_Z$ is the rank-two matrix 
\begin{align}
    T =\begin{bmatrix}
       0&2&1\\
      -1&-1&0\\
      -2&0&1
      \end{bmatrix}.
\end{align}
 The points $e = (-1,1,-2)^\top$ and $e' = (1,2,-1)^\top$ generate the right and left nullspaces of $T$. 
 Pick $\bar{\ell} = e$. Then $e^\top e \neq 0$. Now check that $[e']_\times Y = (T [e]_\times) [e]_\times X$. Indeed, 
      \begin{align} \begin{split}
      [e']_\times Y
      &=\begin{bmatrix}
      0&1&2&3&4&4&5&-1&-4\\
      -1&-1&-2&-2&-3&-2&-3&0&0\\
      -2&-1&-2&-1&-2&0&-1&-1&-4
      \end{bmatrix}\\
      &\sim \begin{bmatrix}
       0&-12&-6&-18&-24&-12&-30&6&18\\
      6&12&6&12&18&6&18&0&0\\
      12&12&6&6&12&0&6&6&18\end{bmatrix}
      = (T [e]_\times) [e]_\times X.
      \end{split}
      \end{align}

 \item Take $x_i$ to be the columns of the matrix $X$ and $y_i$ to be the columns of the matrix $Y$:
 \begin{align}
     X = \begin{bmatrix}
      1&0&0&1&1&1&0&1&2\\
      0&1&0&1&1&0&1&2&-3\\
      0&0&1&1&0&1&1&1&1
      \end{bmatrix}
      \quad \quad 
     Y = \begin{bmatrix}
      1&0&0&1&1&0&1&2&15\\
      0&1&0&1&0&1&1&1&4\\
      0&0&1&1&1&1&0&0&-5
      \end{bmatrix}.
 \end{align}
 The unique element in $\mathcal{N}_Z$ is the rank-three matrix
\begin{align}
    T = \begin{bmatrix}
      0&1&-4\\
      1&0&3\\
      -4&3&0
      \end{bmatrix}.
\end{align}
By construction, $y_i^\top T x_i = 0$ for all $i=1, \ldots, 9$.
\end{enumerate}

\end{ex}

\section{Conclusion}\label{sec:conclusion}
In combination with \cite{rankdrop1}, we now have a complete characterization of how rank deficiency of the matrix $Z=(x_i^\top\otimes y_i^\top)_{i=1}^k$ occurs for all values of $k=2,\ldots,9$. We have also demonstrated a strong correspondence between lines in $\PP(\CC^{3\times 3})$, quadric surfaces in $\PP^3$, and quadratic Cremona transformations of $\PP^2$ under appropriate genericity assumptions, which we have named the {\em trinity correspondence}. We conclude  with a simple corollary of our work that highlights the geometry of reconstructions of semi-generic point pairs of sizes six, seven and eight.

\begin{cor}\label{cor:quadrics are rank drop}
Let $(x_i,y_i)_{i=1}^k\subset\PP^2\times\PP^2$ be semi-generic. Then we get the following:
\begin{itemize}
    \item When $k=6$, $Z_6$ is rank deficient exactly when a reconstruction $p_1,\ldots,p_6,c_1,c_2$ is a Cayley octad (eight points in the intersection of three generic quadrics).
    \item When $k=7$, $Z_7$ is rank deficient exactly when the points $p_1,\ldots,p_7,c_1,c_2$ of any reconstruction lie on a quartic curve that arises as the intersection of two quadrics.
    \item When $k=8$, $Z_8$ is rank deficient exactly when the points $p_1,\ldots,p_8,c_1,c_2$ of any reconstruction lie on a quadric.
\end{itemize}
\end{cor}
\begin{proof}

When $k=8$, the matrix $Z_8$ is rank deficient exactly when $\mathcal{N}_{Z_8}$ is a line. By the semi-genericity of the point pairs, this line is $P$-generic and does not contain any rank one matrices. Any reconstruction of the point pairs corresponds to a fundamental matrix $F$ on this line and by Lemma~\ref{lem:lines through F gives quadric}, the reconstruction lies on a quadric. Similarly, if the point pairs have a reconstruction, given by some fundamental matrix $F$, which lies on a quadric then there is a corresponding line through $F$ in $\mathcal{N}_{Z_8}$ and $Z_8$ is rank deficient.

When $k=7$, $Z_7$ is rank deficient exactly when $\mathcal{N}_{Z_7}$ is a plane. Given any reconstruction $p_1,\ldots, p_7,c_1,c_2$ of the point pairs, let $F$ be the corresponding fundamental matrix.
By semi-genericity of the point pairs, $\mathcal{N}_{Z_7}$ is a generic plane that intersects $\mathcal{D}$ in a curve $C$ of rank-two matrices. If we take any two lines through $F$ in $\mathcal{N}_{Z_7}$ then as in Lemma~\ref{lem:lines through F gives quadric} we obtain two quadrics $Q_1,Q_2$ whose intersection is a quartic curve through the reconstruction.  Similarly, if any reconstruction corresponding to a fundamental matrix $F'$ lies on two distinct quadrics then there are two distinct lines through $F'$ in $\mathcal{N}_{Z_7}$ and $Z_7$ is rank deficient.

For $k=6$, $Z_6$ is rank deficient if and only if 
$\mathcal{N}_{Z_6}$ is a $3$-dimensional plane. Equivalently, every rank- two matrix $F\in\mathcal{N}_{Z_6}$ lies on a net of lines in $\mathcal{N}_{Z_6}$, which corresponds to a net of quadrics containing the reconstruction corresponding to $F$. It follows that if the reconstruction lies on a Cayley octad $Q_1\cap Q_2\cap Q_3$ then $Z_6$ is rank deficient. For the other direction, suppose $Z_6$ is rank deficient. Then the reconstruction lies on a net of quadrics $Q_1\cap Q_2\cap Q_3$ and we need to show that this intersection contains exactly the $8$ points $\{p_i\}_{i=1}^6,c_1,c_2$. If $p'\in Q_1\cap Q_2\cap Q_3$ is any point distinct from $c_1,c_2$, then $\pi_2(p')^\top M\pi_1(p')=0$ for all $M\in\mathcal{N}_{Z_6}$. Due to semi-genericity, the hypothesis of \cite{rankdrop1}[Lemma 6.1] holds for any subset of $5$ point pairs, and it follows that $(\pi_1(p'),\pi_2(p'))=(x_i,y_i)$ for some $i$. We can conclude that $p'=p_i$ and the intersection is indeed a Cayley octad.
\end{proof}

\bibliography{Part2}

\end{document}